\documentclass[12pt]{amsart}
\usepackage[top=1.2in, left=1.2in, right=1.2in, bottom=1.2in]{geometry}
\usepackage[utf8x]{inputenc}
\usepackage{amsfonts}
\usepackage{amssymb,epsfig}
\usepackage{amsmath}
\usepackage{delarray}
\usepackage{graphicx}
\usepackage{color}
\usepackage{tensor}

\newtheorem{thm}{Theorem}[section]
\newtheorem{theorem}[thm]{Theorem}

\newtheorem{lemma}[thm]{Lemma}

\newtheorem{proposition}[thm]{Proposition}
\theoremstyle{definition}
\newtheorem{definition}[thm]{Definition}

\theoremstyle{remark}
\newtheorem{remark}[thm]{Remark}
\numberwithin{equation}{section}
\newtheorem{example}[thm]{Example}

\newcommand{\CC}{\mathbb{C}}

\newcommand{\KK}{\mathbb{K}}

\newcommand{\cL}{\mathcal{L}}

\newcommand{\intprod}{\!\mathbin{\hbox to 6pt
{\vrule height0.4pt width5pt depth0pt \kern-.4pt
\vrule height6pt width0.4pt depth0pt\hss}}\!}

\begin{document}

\title[Systems with singular underlying structures]{Local normal forms of dynamical systems with a singular underlying geometric structure}

\author{Kai Jiang}
\address{Beijing international Center for mathematical research, Peking University}
\email{kai.jiang@bicmr.pku.edu.cn}

\author{Tudor S. Ratiu}
\address{School of Mathematics, Shanghai Jiao Tong
University, 800 Dongchuan Road, Minhang District, Shanghai, 200240 China,
Section de Math\'ematiques, Universit\'e de Gen\`eve, 2-4 rue du
Li\`evre, Case postale 64, 1211 Gen\`eve 4, and Ecole Polytechnique
F\'ed\'erale de Lausanne, CH-1015 Lausanne, Switzerland. Partially
supported by the National Natural Science Foundation of China  grant
number 11871334 and by NCCR SwissMAP grant of the Swiss National Science Foundation. }
\email{ratiu@sjtu.edu.cn}

\author{Nguyen Tien Zung}
\address{Institut de Mathématiques de Toulouse, UMR5219, Université Toulouse 3}
\email{tienzung@math.univ-toulouse.fr}
\thanks{}

\begin{abstract}{ 
We prove, in many different cases, the existence of a simultaneous local 
normalization for couples $(X,\mathcal{G})$, where $X$ is a vector field 
vanishing at a point and $\mathcal{G}$ is a singular underlying 
geometric structure which is invariant with respect to $X$:
singular volume forms, singular symplectic and Poisson structures,
and singular contact structures. Similar  to Birkhoff normalization
for Hamiltonian vector fields, our normalization is also only formal,
in general. However, when $\mathcal{G}$ and $X$ are (real or complex) analytic
and $X$ is analytically   or Darboux integrable then the
simultaneous normalization is also analytic. Our proofs are based
on the toric approach to normalization of dynamical systems, the toric
conservation law, and the equivariant path method. We also consider the case
when $\mathcal{G}$ is singular but $X$ does not vanish at the origin.
}
\end{abstract}

\date{Second version, September 2019}
\subjclass{58K50}
\keywords{normalization, singular geometric structure}%

\maketitle

\begin{small}
\tableofcontents
\end{small}

\section{Introduction}

A classical result due to Birkhoff, Gustavson, and Moser
(see, e.g., \cite{Birkhoff1927,Gustavson-NF1966,Moser-Lectures1968})
states that a Hamiltonian vector field $X$ which
vanishes at a point $O$ admits a normalization \`a la Poincaré-Birkhoff
at $O$ (also in the resonant case), i.e.,
there is a coordinate system $(x_1,\hdots,x_{2n})$
in which $X$ does not contain any non-resonant terms and
the symplectic form has the canonical expression 
$\omega = \sum_{i=1}^n dx_i \wedge dx_{i+n}$.
A similar result, with a similar proof, is known to hold for volume-preserving
vector fields: if $X$ preserves a volume form $\Omega$ then there is a
coordinate system $(x_1,\hdots,x_{n})$ which normalizes $X$ and in
which $\Omega$ has the canonical expression 
$\Omega = dx_1 \wedge \hdots \wedge dx_{n}$. These normalizations of
Hamiltonian and volume-preserving vector fields are only \textit{formal},
in general. However, when $X$ is \textit{integrable} analytically or in
the sense of Darboux, then  there exists a locally
\textit{analytic normalization} \cite{Vey-Isochore1979,Zung-Poincare2002,Zung-Birkhoff2005,
Zung-NFIII2018}.
\medskip

The problem of normalization of Hamiltonian and volume-preserving vector fields
may be viewed as the problem of simultaneous normalization of a
couple $(X,\mathcal{G})$, where $\mathcal{G}$ is an underlying geometric
structure preserved by $X$. (In the above cases, $\mathcal{G}$ is a symplectic
or a volume form.)  In this paper, we  study this simultaneous
normalization problem, when the geometric structure $\mathcal{G}$ itself
is singular at the singular point $O$ of $X$. We may encounter singularities of different types and some singular structures have been studied by different people from various aspects, see, e.g., \cite{Martinet-Formes1970, Roussarie-Local1975,Pelletier-1Form1985,  Zhito-1Form1992,JaZhi-Contact2001,MZ-Poisson2006}.
We  concentrate our attention to the following list of singular geometric
structures, though many other structures (e.g.,
more general Poisson, Dirac, or Nambu structures)  appear in
geometric mechanics and physics:

$\bullet$ Singular volume forms, which either vanish  or blow up
    at the singular point $O$ of $X$ (e.g.,
    $x_1^k dx_1 \wedge \hdots \wedge d x_n$ where $k = \pm 1$).

$\bullet$ Singular symplectic forms, which are either folded
    symplectic or log-symplectic (e.g.,
    $\sum x_i^{k_i} dx_i \wedge dx_{n+i}$ where $k_i \in \{-1,0,1\}$).

$\bullet$ Singular contact structures.

In each of the above cases, we impose
some non-degeneracy or genericity 
condition ensuring that the singular geometric structure itself 
admits a nice normal norm, i.e., has some local \textit{canonical expression}, 
e.g., $\dfrac{1}{x_1} dx_1 \wedge dx_{n+1} + \sum_{i=2}^n dx_i \wedge dx_{n+i}$ 
for a so-called $b$-symplectic structure 
\cite{ GuLi-LogSymplectic2014,GLPR-LogSymplectic2017,
GMP-bSymplectic2014}, or analogous 
expressions for $b^m$-symplectic
and $b^m$-Nambu structures \cite{Miranda-Planas2018Nambu, 
Miranda-Planas2018Equivariant}  (even in the equivariant case, if needed).
Then we show that, in each of these cases, the couple $(X,\mathcal{G})$ 
admits a \textit{simultaneous normalization}, i.e., a coordinate
system which puts $\mathcal{G}$ in some canonical form and $X$ in normal form.

The simultaneous normalization problem is usually quite difficult, especially
when the underlying geometric structure itself is singular. As far as we know,
there are two main methods to study it.
\medskip

The first approach is classical:  one first puts the geometric structure 
$\mathcal{G}$ in canonical form, then makes a step-by-step normalization 
of the vector field $X$ (eliminating resonant nonlinear terms one by one) 
by ``canonical transformations" which leave $\mathcal{G}$ intact. This 
approach was effective for Hamiltonian and volume-preserving systems  and can 
probably be used also for other systems. However, to show the existence of 
the required ``canonical transformations" at each step is a non-trivial task, 
especially when the geometric structure is itself singular. For 
example, when the singular structure is not homogeneous, one cannot find 
canonical transformations order by order directly. In order to leave the 
geometric structure unchanged, one needs to adjust, at each step, the terms 
of the transformation having different orders; this results in very long,   
laborious, and delicate  computations.
\medskip

The second approach, initiated by the third author \cite{Zung-Poincare2002,
Zung-Birkhoff2005,MZ-Poisson2006,Zung-AA2018}, is geometric and is based on the toric 
characterization of the normalization of dynamical systems, together with the
equivariant Moser path method \cite{Moser-Volume1965} (a.k.a. the Lie 
transform method). This is the approach used in this paper and it goes as 
follows.

$\bullet$ For each vector field (or family of commuting vector fields) which
vanishes at a point $O$ there is a unique natural (intrinsic) local 
effective \textit{associated torus $\mathbb{T}^\tau$-action} $\rho$ (in 
the complexified space, if the original space is real) which fixes  $O$. 
The dimension $\tau$ of the torus in question is called the 
\textit{toric degree} of the system at $O$. This action $\rho$ is only formal, 
in general, but when the system is analytically or Darboux integrable, then 
it is automatically analytic.

$\bullet$ \textit{Universal toric conservation law}: any geometric structure 
$\mathcal{G}$ preserved by a dynamical system is also preserved by its 
associated torus action. This is one of our main results
and it applies to many other geometric structures, not just 
the ones studied in this paper; see Theorem \ref{thm:ConservationPropertyIntro} 
below for a precise statement.  

$\bullet$ For the geometric structures studied in this paper, it is known, 
or it can be shown, that $\mathcal{G}$ can be put into
canonical form in some coordinate system (if one forgets about $\rho$),
and it is also known that $\rho$ can be linearized by Bochner's averaging
formula \cite{Bochner-Compact1945} (if one forgets about $\mathcal{G}$).
What needs to be done do is to linearize $\rho$ and put $\mathcal{G}$ in 
canonical form \textit{simultaneously} which is achieved by the use of 
the equivariant Moser path method.

$\bullet$ The linearization of $\rho$ is actually equivalent to the
normalization of the dynamical system. So,   the previous step produces
a coordinate system in which both the geometric structure and the
dynamical system are normalized.

In general, the above steps only give a
\textit{formal} simultaneous normalization, just like in the classical
theory of Poincaré-Birkhoff normalization of vector fields.
The problem of showing the existence, or non-existence, of a local analytic
normalization is, in general, difficult since it involves small divisor
phenomena which are hard to control; see, e.g.,
\cite{Bruno-LocalMethods1989,Ito-Birkhoff1989,Stolovitch-Singular2000}.
However, when our system is analytically   or Darboux integrable,
then the associated torus action is analytic, the path method can also be done 
analytically, and we obtain an analytic simultaneous normalization without
having to deal with small divisors.
\medskip

Our approach (as well as the first classical one) works
mainly in the formal and analytic categories, so our results address
predominantly formal and analytic systems,
though some results are also valid in the smooth
category. (The proofs in the smooth category are usually more
involved and require some specific techniques.)

In this paper, most of the times we do not explicitly distinguish 
between the real  and  complex cases. Our results are valid for both complex  
and real systems, i.e., in the real case the normalization maps are also real,
though we only write explicit proofs for the complex case. The reason 
is that, by local complexification, a real
dynamical system (together with an underlying geometric structure)
can be viewed as a complex dynamical system which is equivariant with respect
to the \textit{anti-complex involution}. As was explained in, e.g.,
\cite{GuilleminSternberg-Linearization1968} and \cite{Zung-Birkhoff2005}, 
even if we deal with a real system,  everything can be done 
in the complexified space by keeping track of the anti-complex involution, 
so that the complex normalization maps in the real case turn out to be
equivariant with respect to the anti-complex involution, i.e., they are, 
in fact, complexifications of real normalization maps.  Of course, some 
``finer" normal forms, e.g., the Jordan normal form for matrices, look 
differently in complex and real coordinates, but this is not the main 
point of our paper.

\medskip

The paper is organized as follows. In Section \ref{section:Preliminary}, 
we collect necessary notions and tools
to be used for our study, including the
toric approach to the normalization problem of vector fields
(Subsection \ref{subsection:TorusActions}), the notions of
analytic and Darboux integrability (Subsection \ref{subsection:Integrable}),
the universal toric conservation law (Subsection \ref{subsection:conservation}), 
and the equivariant Moser path method (Subsection \ref{subsection:PathMethod}). 
The main new result of Section \ref{section:Preliminary} is the following 
theorem,  an important particular case (which has not been proved explicitly 
anywhere, as far as we know) of the universal toric conservation law stating
that ``\textit{anything which is preserved by a dynamical system is also 
preserved by its associated torus actions}'' \cite{Zung-AA2018}.

\begin{theorem}[Theorem \ref{thm:ConservationProperty}]
\label{thm:ConservationPropertyIntro}
Let $X$ be a formal vector field which vanishes at a point $O$ and which 
preserves a formal rational tensor field $\Lambda$ (i.e., $\Lambda = \Omega/f$
where $\Omega$ is a formal tensor field and $f$ is a non-trivial formal 
function). Then the associated torus action of $X$ at $O$ also preserves 
$\Lambda$. If $\Lambda$ is only conformally preserved by $X$ (i.e.,
$\cL_X \Lambda = g\Lambda$ where $g$ is a formal function), then the 
associated torus action of $X$ also conformally preserves $\Lambda$.
\end{theorem}

The case with an invariant singular volume form is treated in Section 
\ref{section:SingularVolume}, with an invariant singular symplectic form
in Section \ref{section:SingularSymplectic}, and with an invariant 
singular contact structure in Section \ref{section:SingularContact}. The 
main results of these last three sections can be put together in the following
big abstract theorem.

\begin{theorem}
Let $X$ be a formal vector field which vanishes at the origin and which 
preserves a singular geometric structure $\mathcal{G}$ which belongs to 
one of the following types: folded volume forms, generic singular Nambu 
structures of top order, (multi-)folded symplectic forms, log-symplectic 
forms, and three different kinds of singular contact structures. Then 
the couple $(X,\mathcal{G})$ can be formally normalized simultaneously. 
Moreover, in the integrable analytic case, when both $\mathcal{G}$ and 
$X$ are analytic and $X$ is analytically or Darboux integrable, then 
there exists a local analytic simultaneous normalization of $(X,\mathcal{G})$.
\end{theorem}

We also treat the cases of couples $(X,\mathcal{G})$ where the geometric
structure is singular at the origin $O$ but $X(O) \neq 0$.
In a series of theorems throughout this paper, we show that,
in such situations, $\mathcal{G}$ can often be put into normal form in a
coordinate system in which the vector field $X$ becomes $\partial/\partial x_k$
where $x_k$  is one of the coordinates. These theorems (for the cases
$X(O) \neq 0$) use more classical step-by-step normalization methods and not
the toric approach for $X$.

\section{Preliminaries}
\label{section:Preliminary}

\subsection{Associated torus actions}
\label{subsection:TorusActions}

In this subsection, we briefly recall from \cite{Zung-Poincare2002, 
Zung-Birkhoff2005, Zung-IntegrableChapter2016, Zung-AA2018} the toric approach
to the problem of local normalization of vector fields.

Denote by $X$ a formal or local analytic vector field
on a manifold, which vanishes at a point $O$. We write the Taylor series of
$X$ in a coordinate system $(x_1,\hdots,x_n)$ around $O$ as
\[
    X= X^s + X^n + \sum_{i \geq 2} X^{(i)},
\]
where $X^{(1)} = X^s + X^n$ is the Jordan decomposition of the linear part 
of $X$ in the above coordinate system, i.e. $X^s$ is its semisimple part 
and $X^n$ is its nilpotent part, and $X^{(i)}$ is the homogeneous term of 
degree $i$ of $X$. We can assume (after a complexification of the system, 
if necessary) that $X^s$ is diagonal:
\[
    X^s= \sum_{i=1}^n \gamma_i x_i \dfrac{\partial}{\partial x_i},
\]
where $\gamma_1, \hdots, \gamma_n \in \mathbb{C}$ are the eigenvalues of 
$X$ at $O$. We say that $X$ is in \textit{normal form} (\`a la 
Poincaré--Birkhoff or Poincaré--Dulac) if
\[
    [X^s, X] = 0,
\]
which amounts to the vanishing of all the non-resonant terms in the Taylor 
expansion of $X$. The classical theorem of Poincaré says that such a formal 
coordinate system always exists, i.e., $X$ can always be normalized formally.

We may view $X$, via its Lie derivative,
as a linear differential operator on the space of formal functions.
Then this operator, even though acting on an infinite dimensional space, 
also admits an intrinsic unique Jordan decomposition
\[
    X = X^S + X^N
\]
where $X^S$ and $X^N$ are formal vector fields (also viewed as linear operators)
which are intrinsic (i.e., they depend on $X$ but do not depend on the choice 
of local coordinates). The vector field $X$ is in normal form if and only if 
$X^S$ is linear semisimple, i.e. $X^S = X^s$ and $X^N = X^n + X^{(2)} + \cdots$.

The smallest number $\tau \geq 0$ such that we can express $X^S = X^s$
(in normalized coordinates) in the form
\begin{equation}\label{eq:Diagonal.VFa}
    X^S = \sum_{i=1}^\tau \rho_i  Z_i,
\end{equation}
where $\rho_i \in \mathbb{C}$ are constants and
\begin{equation}\label{eq:Diagonal.VF}
    Z_i = \sum_{j=1}^n a_{ij} x_j \dfrac{\partial}{\partial x_j}
\end{equation}
are  diagonal vector fields with integer coefficients $a_{ij} \in \mathbb{Z}$,
is called the \textit{\textbf{toric degree}} of $X$ at $O$. The minimality 
condition on $\tau$ is equivalent to the condition that 
$\rho_1,\hdots,\rho_\tau$ are incommensurable.
The vector fields $\sqrt{-1}Z_1,\hdots, \sqrt{-1}Z_\tau$ are the generators
of an effective  torus $\mathbb{T}^\tau$-action $\rho$, which is uniquely 
determined by $X$ (up to automorphisms), and which is called the intrinsic 
\textit{\textbf{associated torus action}} of $X$ at $O$.

In general, even when $X$ is analytic, its normalization is only formal, and 
its associated torus action $\rho$ is also formal. Thanks to Bochner's 
averaging formula \cite{Bochner-Compact1945}, any compact group action (even 
formal ones) can be linearized, and the normalization of $X$ is the same as 
the linearization of its associated torus action $\rho$ (i.e., \textit{$X$ 
is in normal form if and only if $\rho$ is linear}). The  vector field 
$X$ admits a local analytic normalization if and only if the action
$\rho$ is analytic. So this torus action $\rho$ plays a very important role 
in normalization problems  and it admits the \textit{fundamental conservation 
property} \cite{Zung-AA2018} which will be discussed
in Subsection \ref{subsection:conservation}.

\subsection{Integrable systems}
\label{subsection:Integrable}

Recall  that a \textit{\textbf{Darboux-type function}} is a (generally 
multi-valued) function $F$ which can be written as
$$
F = \prod_{k=1}^m F_k^{c_k},
$$
where $F_k$ are analytic functions and $c_k$ are complex numbers. The class 
of Darboux-type functions on a manifold is significantly larger than the 
class of analytic or rational functions. Note that, even though 
$F = \prod_{k=1}^m F_k^{c_k}$ is multivalued, its logarithmic differential
$$ \dfrac{dF}{F} = d(\ln F) = \sum _{k=1}^m c_k \dfrac{dF_k}{F_k}$$
is a single-valued rational differential 1-form,  so it makes  
sense to talk about Darboux-type first integrals of vector fields:
$X(\prod_{k=1}^m F_k^{c_k}) = 0$ means that
$\sum _{k=1}^m c_k \dfrac{X(F_k)}{F_k} = 0$.
In many problems in dynamical systems, one may be interested in Darboux-type 
first integrals when there do not exist enough analytic or rational first 
integrals (see., e.g., \cite{Zhang-Integrability2017}).

In this paper, for analytic normalization, we will consider dynamical systems 
which are integrable in the usual non-Hamiltonian sense (see, e.g., 
\cite{Bogoyavlenskij-Integrable1996, Stolovitch-Singular2000, Zung-Poincare2002, 
Zung-AA2018}), also in the case when the first integrals are only Darboux-type 
functions. Let us recall here the definition of integrability.

\begin{definition}
i) A dynamical system given by a vector field $X$ on an $n$-dimensional
manifold $M$ is called \textbf{\textit{integrable}} if there exist $p$
vector fields $X_1=X$, $X_2,\ldots,X_p$ and $q$ functions $F_1,\ldots,F_q$ 
on $M$, such that the vector fields commute pairwise, 
$X_1\wedge\cdots\wedge X_p\neq0$ almost everywhere, the functions are common 
first integrals for these vector fields, and 
$\mathrm dF_1\wedge\cdots\wedge\mathrm dF_q\neq0$ almost everywhere.
The integers $p,q$ satisfy $p\geqslant1,q\geqslant0,p+q=n$.

ii) An $n$-tuple $(X_1,\ldots,X_p,F_1,\ldots,F_q)$ such as above is also
 called an \textit{\textbf{integrable system of type $(p,q)$}}.

iii) If the vector fields $X_1,\ldots,X_p$ and the functions $F_1,\ldots,F_q$
are all smooth (resp. formal, resp. analytic) then the vector field $X$ and the 
system $(X_1,\ldots,X_p,F_1,\ldots,F_q)$ are called \textit{\textbf{smoothly}}
(resp. \textit{\textbf{formally}}, resp. \textit{\textbf{analytically}}) 
\textbf{\textit{integrable}}.

iv) If the vector fields $X_1,\ldots,X_p$ are rational (i.e. can be written 
as the quotient of an analytic vector field by a non-trivial analytic function), 
and the functions  $F_1,\ldots,F_q$ are Darboux-type functions, then $X$ and 
the system $(X_1,\ldots,X_p,F_1,\ldots,F_q)$ are called 
\textit{\textbf{Darboux integrable}}.
\end{definition}

The following result, obtained by the third author via geometric approximation 
methods (instead of the fast convergence method), will allow us to analytically 
normalize integrable systems.

\begin{theorem}[\cite{Zung-Poincare2002,Zung-Birkhoff2005,Zung-NFIII2018}]
Let $X$ be a local analytic vector field which vanishes at a point $O$ and 
which  is analytically integrable or Darboux integrable. Then the associated 
torus action of $X$ at $O$ is  locally analytic.
\end{theorem}

\subsection{The toric conservation law}
\label{subsection:conservation}

The \textit{toric conservation law}, which states that \textit{anything which 
is preserved by a dynamical system is also preserved by its associated torus 
action}, was first discovered in its general form by the third author
\cite{Zung-AA2018}, though some particular cases of this law
(e.g., when that ``anything" is a function, i.e., a first integral) have 
been known before. Here, the system may be given by a vector field, or a 
discrete-time system, or a quantum system, or a stochastic system, etc.; 
``anything" may be a tensor field, or a subbundle of a natural bundle over the 
manifold, or a differential operator, etc., which may be smooth, analytic, 
meromorphic, formal, Darboux-type, etc.. The associated torus actions 
in question depend on the situation: for example, for integrable 
Hamiltonian systems near a regular Liouville torus,   this torus action is 
the Liouville torus action, whereas if the system is a local dynamical system 
which vanishes at a point, then this torus action is the one explained 
in Subsection  \ref{subsection:TorusActions}.

For each particular situation, the toric conservation law can be stated as 
a rigorous theorem or conjecture. There are still many situations for which no 
proof of the corresponding conjectures have been explicitly written.
In this section we consider such a situation, which is useful for 
our simultaneous normalization problems.

Consider a formal vector field $X$ on $\mathbb{C}^n$ which vanishes at the 
origin and let $\Lambda$ be a formal rational tensor field  on $\mathbb{C}^n$,
i.e., $\Lambda$ is the quotient of a formal tensor field $\Omega$ by a 
nontrivial formal function $f$: $\Lambda = \dfrac{\Omega}{f}$. We   say 
that $\Lambda$ is an \textit{\textbf{invariant}} of $X$ if it is 
\textit{\textbf{conserved}} by $X$, i.e., $\cL_X\Lambda=0$. We   say 
that $\Lambda$ is a \textit{\textbf{semi-invariant}} of $X$ if it is 
\textit{\textbf{conformally conserved}} by $X$, i.e., there is a formal 
function $g$ such that $\cL_X\Lambda= g\Lambda$.

Observe that the set of nontrivial semi-invariant formal rational functions 
of $X$ form a multiplicative group: if
$\cL_X F_1 = g_1 F_1$ and $\cL_X F_2 = g_2 F_2$
then $\cL_X (1/F_1) = - g_1 (1/F_1)$ and
$\cL_X (F_1F_2) = (g_1+g_2) F_1F_2$.
Moreover, if $\Lambda = \dfrac{\Omega}{f}$ is a formal rational tensor field 
written in reduced form, i.e. $\Omega$ and $f$ are co-prime,
and $\Lambda$ is a semi-invariant tensor field of $X$, then both
$\Omega$ and $f$ are semi-invariants of $X$.
Indeed, $\cL_X ( \Omega/f ) = g  \Omega/f$ means
that $f\cL_X(\Omega) - \cL_X(f) \Omega = gf\Omega$, or
$\cL_X(f)\Omega = f(\cL_X\Omega - g\Omega)$. Since $\Omega$ is co-prime with 
$f$, it implies that $\cL_X(f)$ is divisible by $f$, i.e., 
$\cL_X(f) = hf$, where $h$ is a formal function, and then we have 
$\cL_X\Omega = (g+h)\Omega$.

\begin{theorem}[Conservation Theorem]
\label{thm:ConservationProperty}
 Let $X$ be a formal vector field on $\mathbb{C}^n$ which vanishes at the 
 origin.  Assume that $X$ is already in Poincaré-Dulac normal form  and that  
 $(Z_1,\ldots,Z_\tau)$ is a $\tau$-tuple
 of diagonal vector fields which generate the associated torus action of $X$ 
 as in \eqref{eq:Diagonal.VFa}, \eqref{eq:Diagonal.VF}. Let
 $\Lambda$ be a formal rational tensor field on $\mathbb{C}^n$.
 Then we have:

{\rm (i)} If $\Lambda$ is preserved by $X$, i.e.,  $\cL_X\Lambda=0$,
 then $\Lambda$ is also preserved by the associated torus
 action of $X$, i.e., $\cL_{Z_k}\Lambda=0$ for every
 $k=1,\hdots, \tau$.

{\rm (ii)} If $\Lambda$ is conformally preserved by $X$, i.e.,
 $\cL_X\Lambda=g\Lambda$, where $g$ is a formal function, then
 $\Lambda$ is also conformally preserved by the
 associated torus action of $X$, i.e., $\cL_{Z_k}\Lambda=g_k\Lambda$ for every 
 $k=1,\hdots, \tau$, where $g_1,\hdots, g_\tau$ are formal functions.
\end{theorem}

In order to prove the above theorem, we recall the following result due to
Walcher (Lemma 2.2 in \cite{Walcher-Poincare2000}, see also Lemma 2.1 
of \cite{Zung-NFIII2018}).

\begin{lemma} \label{lemma:Walcher}
Let $X$ be a formal vector field in  Poincaré-Dulac normal form. Assume 
that $F=\sum_{k}F^{(k)}$ is a formal semi-invariant invariant of $X$, i.e., 
$X(F)=\lambda F$ for some formal function $\lambda$.  Then there exists an 
invertible formal function $\beta$ (i.e., $\beta(0) \neq 0$)
such that $\tilde F=\beta F$ is a semi-invariant of both $X$
and of its semisimple part $X^s$. Moreover, the formal
function $\tilde \lambda$, satisfying $X(\tilde F)=\tilde\lambda \tilde F$, is 
a first integral of $X^s$ and we have $X^s(\tilde F)=\lambda(0) \tilde F$. 
In addition,  $Z_k(\tilde F)=c_k \tilde F$ for every $k=1,\hdots,\tau$, where 
$c_k$ are some complex numbers and $Z_1,\hdots,Z_\tau$ are the generators given 
in formulas \eqref{eq:Diagonal.VFa} and \eqref{eq:Diagonal.VF} of the 
associated torus action of $X$.
\end{lemma}

We now turn to the proof of the Conservation Theorem 
\ref{thm:ConservationProperty}.

\begin{proof}
(i) We can write the formal rational tensor field $\Lambda$
 as $\Lambda =  \Omega/f$ where $f$ is a formal
 function and $\Omega$ is a formal tensor field co-prime with
 $f$. As observed  above, $\cL_X(\Omega/f)=0$ implies that
 $f$ and $\Omega$ are  semi-invariants of $X$, i.e.,
 \begin{equation}\label{eq2.7}
     \cL_X\Omega=\lambda\Omega\quad \text{and}\quad X(f)=\lambda f
 \end{equation}
for some formal function $\lambda$. Invoking Walcher's Lemma 
\ref{lemma:Walcher}, we can assume $\lambda$ is a first integral of $X^s$ 
(by multiplying both $\Omega$ and $f$ by a formal function of the type 
$1+\beta$, if necessary), that $X^{s}(f)=\lambda(0)f$, and that $Z_k(f)=c_k f$ 
with $c_k\in\CC$ for $k=1,\ldots,\tau$.

We want to show that $\cL_{X^s}( \Omega/f)=0$  or, equivalently, 
$f\cL_{X^s}\Omega=X^{s}(f)\Omega$. Since $X^{s}(f) = \lambda(0)f$, the 
equation to prove is reduced to $\cL_{X^s}\Omega=\lambda(0)\Omega$. We 
prove it by induction on the degree of the terms in \eqref{eq2.7}.

Write $\lambda=\lambda^{(0)}+\lambda^{(1)}+\cdots$ and 
$\Omega=\Omega^{(0)}+\Omega^{(1)}+\cdots$ in which the upper right index
$(k)$ indicates the homogeneous part of degree $k$ of the corresponding term. 
Obviously, $\cL_{X^{(1)}}\Omega^{(0)}=\lambda(0)\Omega^{(0)}$,
therefore $\cL_{X^{s}}\Omega^{(0)}=\lambda(0)\Omega^{(0)}$. Now suppose 
$\cL_{X^{s}}\Omega^{(k)}=\lambda^{(0)}\Omega^{(k)}$ for all $k<\ell$.
Look at the terms of degree $\ell$ in \eqref{eq2.7}; we have
 \begin{equation}\label{eq2.8}
     \cL_{X^{(1)}}\Omega^{(\ell)}+\sum_{k=2}^{\ell+1}\cL_{X^{(k)}}
     \Omega^{(\ell+1-k)}=\lambda^{(0)}\Omega^{(\ell)}+\sum_{k=1}^{\ell}
     \lambda^{(k)}\Omega^{(\ell-k)}.
 \end{equation}
Let $X^s$ act on both sides of the above equation. Since $[X^s,X^{(k)}]=0$ 
and $X^s(\lambda^{(k)})=0$, we have $\cL_{X^s}\Omega^{(k)}=
\lambda(0)\Omega^{(k)}$ for all $k$ by Lemma \ref{lemma:Walcher}, and so we get
\begin{equation}\label{eq2.9}
\cL_{X^{(1)}}\cL_{X^{s}}\Omega^{(\ell)}+\lambda^{(0)}\sum_{k=2}^{\ell+1}
\cL_{X^{(k)}}\Omega^{(\ell+1-k)}=\lambda^{(0)}\cL_{X^s}
\Omega^{(\ell)}+\lambda(0)\sum_{k=1}^{\ell}\lambda^{(k)} \Omega^{(\ell-k)}.
\end{equation}
Take the difference between the expression \eqref{eq2.9} and 
$\lambda^{(0)}$ times the expression \eqref{eq2.8} to get
 \[
      \cL_{X^{(1)}}\left(\cL_{X^s}\Omega^{(\ell)}-\lambda^{(0)}
     \Omega^{(\ell)}\right)=\lambda^{(0)}(\cL_{X^s}\Omega^{(\ell)}
     -\lambda^{(0)}\Omega^{(\ell)});
 \]
therefore
  \begin{equation}\label{eq:2.11}
     \cL_{X^s}\left(\cL_{X^s}\Omega^{(\ell)}-\lambda^{(0)}
     \Omega^{(\ell)}\right)=\lambda^{(0)}(\cL_{X^s}
     \Omega^{(\ell)}-\lambda^{(0)}\Omega^{(\ell)}).
 \end{equation}
This means that $\cL_{X^s}\Omega^{(\ell)}-\lambda^{(0)}\Omega^{(\ell)}$ is 
an eigenvector with eigenvalue $\lambda^{(0)}$ of the linear operator 
$\cL_{X^s}$ on the space $E$ of homogeneous tensor fields of degree $\ell$.
We now decompose $E$ into a direct sum of eigenspaces  and  write 
$\Omega^{(\ell)}=\sum_j\Omega^{(\ell)}_j$ where $\Omega^{(\ell)}_j$ is 
a vector in the eigenspace $E_{c_j}$ with eigenvalue $c_j$. Restricting
attention to $E_{c_j}$, we can easily conclude $c_j=\lambda^{(0)}$ by 
\eqref{eq:2.11}. Thus $c_j=\lambda^{(0)}$ for all $j$,
and $\Omega^{(\ell)}$ itself lies in the eigenspace $E_{\lambda^{(0)}}$,
that is, $\cL_{X^s}\Omega^{(\ell)}=\lambda^{(0)}\Omega^{(\ell)}$.
By induction, we conclude that $\cL_{X^s}\Omega=\lambda^{(0)}\Omega$.

Each monomial tensor field $\Theta$ is an eigenvector of the
Lie derivative operator of the diagonal
vector field $X^s = \sum_{k=1}^\tau \rho_k Z_k$, i.e., we have
$\cL_{X^s} \Theta  = \sum_{k=1}^\tau \rho_k c_k(\Theta) \Theta$ for some 
numbers $c_k(\Theta)$ which, a priori, depend on $\Theta$. Since 
$\cL_{X^s}\Omega=\lambda^{(0)}\Omega$, we must have 
$\sum_{k=1}^\tau \rho_k c_k(\Theta) = 
\lambda(0)$ for every monomial term $\Theta$ of
$\Omega$. Since the numbers $\rho_1,\hdots,\rho_\tau$ are incommensurable, 
there is at most one $\tau$-tuple of complex numbers $c_1,\hdots, c_k$ 
satisfying the equation $\sum \rho_k c_k = \lambda(0)$,
which means that $c_k = c_k(\Theta)$ does not depend on $\Theta$  and 
that $\cL_{Z_k}(\Omega) = c_k \Omega$. Similarly, we have
$\cL_{Z_k}(f) = c_k f$  and hence $\cL_{Z_k}(\Omega/f) = 0$.
\smallskip

(ii) The proof of the second part consists of the following 3 steps:

\underline{Step 1.} Write $\Lambda =  \Omega/f$,
where $\Omega$ and $f$ are co-prime. Then both $\Omega$ and $f$ are
semi-invariants of $X$. By Lemma \ref{lemma:Walcher}, we conclude that $f$ is a 
semi-invariant of the associated torus action of $X$, so it is enough to show
that $\Omega$ is also a semi-invariant of the associated torus action of $X$,
i.e., the problem is reduced to the case when $\Lambda = \Omega$
is a formal tensor field.

\underline{Step 2.} Write $\cL_X\Omega = \lambda\Omega$,
where $\lambda$ is a formal function. Putting $\Omega^* = \beta\Omega$,
where $\beta$ is an appropriate invertible formal function,
we can arrange so that  $\cL_{X^s}\Omega^* = \lambda(0) \Omega^*$. The proof 
of this step is the same as the proof of Walcher's Lemma \ref{lemma:Walcher}. 
(See the proof of Lemma 2.1(i) in \cite{Zung-NFIII2018}.) To keep the
paper self-contained, due to the importance of the Conservation Theorem 
\ref{thm:ConservationProperty}, we give below the full proof.

The semi-invariance of $\Omega$ with respect to $X$ is equivalent to
\begin{equation*}
\cL_{X^s}(\Omega^{(r+j)}) + \cL_{X^n}(\Omega^{(r+j)}) + 
\cL_{X^{(2)}}(\Omega^{(r+j-1)}) + \hdots + \cL_{X^{(j+1)}}(\Omega^{(r)}) 
= \lambda^{(0)} \Omega^{(r+j)} +  \hdots + \lambda^{(j)} \Omega^{(r)}
\end{equation*}
for all $j \geq 0$.

Note that $X^s(\lambda^{(0)}) = 0$.
Now assume that $X^s(\lambda^{(j)}) = 0$
for all $j < k$  and let $\tilde{\Omega} = (1 + \beta_k)\Omega$, where
$\beta_k$ is some homogeneous function of degree $k$. Then
\begin{equation*}
\tilde{\Omega}  = \Omega^{(r)} + \cdots + \Omega^{(r+k-1)} + 
(\Omega^{(r+k)} + \Omega^{(r)} \beta_k) + \cdots
\end{equation*}
and
\begin{equation*}
\cL_X(\tilde{\Omega}) = \tilde{\lambda}\tilde{\Omega},
\end{equation*}
with
\begin{equation*}
\tilde{\lambda} = \lambda^{(0)} + \cdots + \lambda^{(k-1)} + 
(\lambda^{(k)} + X^{(1)}(\beta_k)) + \cdots .
\end{equation*}
Due to the semi-simplicity of $X^s$, one can choose $\beta_k$ such that
$X^s(\lambda^{(k)} + X^{(1)}(\beta_k)) = 0$. Thus, the  assertion is proved by
induction on $k$: $\beta$ can be constructed in the form of an infinite
product $\prod_{k=1}^\infty (1 + \beta_k)$, where each $\beta_k$ is homogeneous 
of degree $k$. (Such an infinite product converges in the space of formal  power 
series). From $X^s(\lambda^*) = 0$ one deduces that $\cL_{X^s}(\Omega^*) = 
\lambda^{(0)} \Omega^*$, again by induction and by the semi-simplicity of $X^s$.

\underline{Step 3.} The equation $\cL_{X^s}(\Omega^*) = \lambda^{(0)} \Omega^*$
implies that $\Omega^*$ (and hence $\Omega$) is a semi-invariant of
the vector fields $Z_k$, $k=1,\hdots,\tau$. This step is already done
at the end of the proof of (i).
\end{proof}

\subsection{Equivariant Moser path method}
\label{subsection:PathMethod}

We recall the well-known path method, which
was introduced by Moser in \cite{Moser-Volume1965}
to show  the equivalence of two volume forms, and
which can be generalized to many other situations in order to
show  the equivalence of two tensor fields on a manifold (see, e.g., 
\cite[Section 5.4]{AbMaRa1988} for a general presentation).

 Let two tensor fields
$\mathcal{G}_0$ and  $\mathcal{G}_1$  be given on a smooth manifold  $M$.
We say they are \textit{\textbf{locally  equivalent}} at  $O \in  M$
if there is a diffeomorphism $\varphi$   of one  neighborhood of  $O$ to 
another neighborhood of $O$, such that
$\varphi^\ast \mathcal{G}_1 = \mathcal{G}_0$.
 One way to show that $\mathcal{G}_0$   and
$\mathcal{G}_1$  are locally equivalent is to join them by
a  curve of tensor fields $\mathcal{G}(t)$
satisfying $\mathcal{G}(0) = \mathcal{G}_0$,
$\mathcal{G}(1) = \mathcal{G}_1$ and to seek a curve  of
local diffeomorphisms  $\varphi_t$   such that
$\varphi_0 = Id$   and
\[
\varphi_{t} ^{\ast} \mathcal{G}(t) = \mathcal{G}_0 ,
\quad \forall t \in [0,1].
\]
Then $ \varphi  = \varphi_1 $ is the desired
diffeomorphism.

A way to find the curve of diffeomorphisms
$\varphi_t$   satisfying the  relation above is to  solve
the equation
\begin{equation}
 \label{eqn:Moser_path}
\mathcal{L}_{X_{t}} \mathcal{G}(t) +
\dfrac{d}{dt} \mathcal{G}(t) =  0
\end{equation}
for a smooth time-dependent  vector field $X_t$.  If this is possible, 
let  $\varphi_t  = F_{t,0}$,  where $F_{t,s}$ is the evolution
operator of the  time-dependent vector field  $X_t$.  Then  we  have
\[
\dfrac{d}{dt}\varphi_t^{\ast}\mathcal{G} (t) =
\varphi_t ^{\ast} \left(\mathcal{L}_{X _t}
\mathcal{G}(t) + \dfrac{d}{d
t}\mathcal{G}(t) \right) = 0,
\]
so that $\varphi_t^{\ast}\mathcal{G}(t) = \varphi_0 ^{\ast}
\mathcal{G}(0) = \mathcal{G}_0 $.  If we choose $X_t$ such that  $X_t(O)
= 0$, then $\varphi_t$ exists for a time $t\geq 1$ in an
open neighborhood of $O$  and
$\varphi_t(O) = O$.

One often takes  $\mathcal{G}(t) = (1 - t)\mathcal{G}_0 +  t\mathcal{G}_1$.
Also, in  applications, this method is not always used in exactly
this way since the  algebraic equation for $X_t$  might be
hard to solve. In such situations the spirit of the path method is used 
(as, e.g., in the proof of the Frobenius Theorem given in 
\cite[Section 4.4]{AbMaRa1988}).
\medskip

Now suppose that a compact Lie group $G$ acts smoothly on $M$  preserving
both tensor fields $\mathcal{G}_0$ and $\mathcal{G}_1$, i.e., 
$\rho_g^* \mathcal{G}_i = \mathcal{G}_i$, for all $g \in G$, $i=0,1$, where 
$\rho: G \times M \rightarrow M$ is the $G$-action. Choose the path of 
tensor fields $\mathcal{G}(t)$ to be also invariant under the action $\rho$;
this condition is clearly satisfied if $\mathcal{G}(t) = (1-t)\mathcal{G}_0 
+ t \mathcal{G}_1$. Let $X_t$ be a solution of \eqref{eqn:Moser_path} and 
define the $G$-averaged smooth vector field
\[
Y_t = \int_G \rho_g^*X_t d\mu,
\]
where $\mu$ is the Haar measure on $G$. By invariance of the Haar measure 
relative to group translations, $Y_t$ is $G$-invariant. In addition, since 
$\rho_g^* \mathcal{G}(t) = \mathcal{G}(t)$ for any $g \in G$ and 
$t \in [0,1]$, we have
\begin{align*}
\mathcal{L}_{Y_t}\mathcal{G}(t) +
\dfrac{d}{dt}\mathcal{G}(t) &=
\int_G \left(\mathcal{L}_{\rho_g^*X_t}
\mathcal{G}(t)\right) d\mu + \dfrac{d}{dt}\mathcal{G}(t)\\
&= \int_G \left(\mathcal{L}_{\rho_g^*X_t} \rho_g^*\mathcal{G}(t) + 
\dfrac{d}{dt}\rho_g^*\mathcal{G}(t)\right) d\mu \\
&= \int_G \rho_g^* \left(\mathcal{L}_{X_t}
 \mathcal{G}(t) + \dfrac{d}{dt} \mathcal{G}(t)
\right) d\mu = 0,
\end{align*}
and hence $Y_t$ also solves \eqref{eqn:Moser_path}.  This shows that 
$\mathcal{G}_0$ and $\mathcal{G}_1$ are $G$-equivariantly equivalent.
 We proved the following result.

\begin{lemma}
\label{lem:EquivariantPath}
Suppose that the compact Lie group $G$ acts locally
smoothly (respectively formally or analytically) on $M$,
preserving two tensor fields $\mathcal{G}_0$ and $\mathcal{G}_1$. If 
\eqref{eqn:Moser_path} has a smooth
(respectively formal  or analytic) solution
$X_t$, then the two structures are locally smoothly (respectively formally 
or analytically) $G$-equivariantly equivalent.
\end{lemma}

\section{Systems with an invariant singular volume form}
\label{section:SingularVolume}

In this section, we work with a local volume form
$$
\Omega = f(x_1,\hdots,x_n) dx_1 \wedge \hdots \wedge dx_n
$$
which is singular in one of the following two senses:
either $f$ is a formal or analytic function which
vanishes at the origin $O = (0,\hdots,0)$, or $f$
blows up at $O$. When $f$ blows up at $O$, it is more convenient to  look at
the dual $n$-vector field (i.e., a contravariant volume form, a.k.a.
Nambu structure of top order) $\Lambda = \Omega^{-1}$ defined by
$\langle \Lambda, \Omega \rangle = 1$, i.e,
$$\Lambda = g(x_1,\hdots,x_n) \dfrac{\partial}{\partial x_1}
\wedge \hdots \wedge \dfrac{\partial}{\partial x_n},$$
where $g = 1/f$ is now a formal or analytic function
which vanishes at $O$.

We assume that $O$ is a \textit{non-degenerate} singular point,
i.e., we have the following two cases:
\begin{itemize}
    \item \underline{Case 1}. $f(O) = 0$ and $df(O) \neq 0$. In this case
$\Omega$ is called a \textit{folded volume form} (because it can be 
obtained as the pull-black of a usual volume form by a fold map, like in 
the case of folded symplectic structures \cite{Melrose-Folded1981}).
This case is investigated in Subsection \ref{subsection:FoldedVolume}.
    \item \underline{Case 2}. $g(O) = 0$ and $dg(O) \neq 0$.
In this case $\Omega$ is called a non-degenerate \textit{log-volume form} 
and its dual $n$-vector field $\Lambda$ is a non-degenerate singular
Nambu structure of top order. This case is investigated in Subsection   
\ref{subsection:LogVolume}.
\end{itemize}

\subsection{The case of a folded volume form}
\label{subsection:FoldedVolume}

The following lemma about the canonical form of a folded volume form
is a well-known folklore result, but we don't know any exact reference for it,
so we present its proof.

\begin{lemma}
\label{lem:NFforOmega}
Consider a local differential form of top order
$$\Omega = f dx_1 \wedge \hdots \wedge dx_n$$ on an $n$-dimensional
manifold $(n \geq 2),$
where $f$ is a formal (resp. local analytic, resp. smooth)
function such that $f(O) = 0$  and $df(O) \neq 0$. Then there exists a
formal (resp. analytic, resp. smooth) coordinate system
$(y_1,\hdots, y_n)$ around $O$
such that
$$\Omega = y_1 dy_1 \wedge \hdots \wedge dy_n$$
\end{lemma}

\begin{proof}
Denote by $S =\{f =0\}$ the hypersurface of singular points of $\Omega$.
By the implicit function theorem, locally we can write
$S = \{x_1 = h(x_2,\hdots,x_n)\}$ where $h$ is some formal (or analytic,
or smooth) function. Put $\hat{x}_1 = x_1 - h(x_2,\hdots,x_n)$.
Then $$\Omega = f d\hat{x}_1 \wedge dx_2 \wedge \hdots \wedge dx_n$$ and $f = 0$
at $\hat{x}_1 = 0$. Thus, by this change of variables, we may assume that
$S =\{x_1 =0\}$, i.e, $f$ is divisible by $x_1$: $f = x_1 \phi$, where
$\phi$ is a function such that $\phi(O) \neq 0.$
Integrating $f = x_1 \phi$ in the direction of $\dfrac{\partial }{\partial x_1}$,
we obtain a function
$$ g(x_1,\hdots,x_n) = \int_0^{x_1}t \phi(t,x_2,\hdots,x_n) dt$$
satisfying
$$\Omega =  dg \wedge dx_2 \wedge \hdots \wedge dx_n,$$
and such that $g(0,x_2,\hdots,x_n) = 0$, 
$\dfrac{\partial g}{\partial x_1}(0,x_2,\hdots,x_n) = 0$ but 
$\dfrac{\partial^2 g}{\partial^2 x_1}(O) = \phi(O) \neq 0$, which means 
that $g$ can be written as $g = x_1^2 h$ where $h(0) \neq 0.$
Put $y_1 = x_1 \sqrt{2h}$ (so that $g = y_1^2/2$), 
$y_2 = x_2, \hdots, y_n = x_n$, so
we get a coordinate system $(y_1,\hdots,y_n)$ with
$\Omega =  d(y_1^2/2) \wedge dy_2 \wedge \hdots \wedge dy_n =
y_1 dy_1 \wedge \hdots \wedge dy_n$. (In the real case with $h(O) < 0$,
put  $y_1 = x_1 \sqrt{-2h}, y_2 = - x_2, y_3 = x_3, \hdots, y_n = x_n$.)
\end{proof}

Let $X$ be a vector field
which preserves a folded volume form $\Omega$.
If $X(O) \neq 0$ then it is easy to see that $X$ can be rectified
together with $\Omega$, i.e., there is a coordinate system
$(y_1,\hdots, y_n)$ in which
$$\Omega = y_1 dy_1 \wedge \hdots \wedge dy_n\quad  \text{and}\quad
X =\dfrac{\partial}{\partial y_n}.$$
Indeed, consider the local $(n-1)$-dimensional
quotient space of the local flow generated by $X$.
The local $(n-1)$-form $X \intprod \Omega$ is
the pull-back of a local folded volume form on this local
quotient space so we can write
$X \intprod \Omega =  y_1 dy_1 \wedge \hdots \wedge dy_{n-1}$
by the previous Lemma \ref{lem:NFforOmega}. Now take any function $y_n$ 
such that $X(y_n) = 1$ and $y_n(O) = 0$ and we are done.

Consider now the case when $X(O) = 0$. We have the following
theorem, which is a generalization of the results about normalization
of isochore vector fields (see \cite[Section 4]{Zung-Poincare2002})
from the case of a regular volume form to the case of a folded volume form.

\begin{theorem}
\label{thm:NFFoldedVolume}
Let $X$ be a formal vector field which vanishes at a point $O$ and
which preserves a formal folded volume form $\Omega$. Then
$X$ can be formally normalized together with $\Omega$: $\Omega$
has canonical form and $X$ is in normal form in the coordinate system 
$(x_1,\hdots,x_n)$ and the
semisimple part of $X$ is diagonal in this system (over $\mathbb{C}$):
$$
\Omega = x_1 dx_1 \wedge \hdots \wedge dx_n,\quad X^S = 
\sum \gamma_i x_i \dfrac{\partial}{\partial x_i}.
$$
The eigenvalues of $X$ satisfy the following resonance relation:
\begin{equation}
\label{eqn:ResonanceFoldedVolume}
 2\gamma_1 + \gamma_2 + \hdots + \gamma_n = 0.
\end{equation}
Moreover, if $X$ and $\Omega$ are analytic and $X$ is analytically
integrable or Darboux integrable, then there exists such a normalization
which is locally analytic.
\end{theorem}

\begin{proof}
By Lemma \ref{lem:NFforOmega} we may write
$\Omega = y_1 dy_1 \wedge \hdots \wedge dy_n$ in some coordinate system.
Denote by $\rho$ the associated torus $\mathbb{T}^\tau$-action of $X$ at $O$,
where $\tau$ is the toric degree of $X$ at $O$.
 Theorem \ref{thm:ConservationProperty} ensures that $\rho$ preserves $\Omega$.
This implies that the singular set
$S = \{y_1 = 0\}$ of $\Omega$ is  a hyperplane which is preserved by $\rho$.
Since $\Omega$ is homogeneous in the coordinate system $(y_1,\hdots,y_n)$,
the linear part of the action $\rho$ in this coordinate system, which is
a linear torus action denoted by $\rho_1$, also preserves
$\Omega$. Notice that, since $S = \{y_1 = 0\}$ is a hyperplane preserved
by $\rho$, it is also preserved by $\rho_1$.

By Bochner's averaging formula \cite{Bochner-Compact1945},
we find a formal diffeomorphism $\Phi$, whose linear part is the identity,
and which intertwines $\rho$ with
$\rho_1$, i.e.,
$$\Phi \circ \rho_1(s,.)  = \rho (s, .) \circ \Phi, \quad \forall  
s \in \mathbb{T}^\tau.$$
Let
$$
\Omega_1 = \Phi^* \Omega.
$$
Then $\rho_1$ also preserves $\Omega_1$: $\rho_1 (s,.)^* \Omega_1 =
\rho_1 (s,.)^* \Phi^* \Omega = (\Phi \circ \rho_1 (s,.))^* \Omega
= (\rho (s, .) \circ \Phi)^* \Omega  =  \Phi^* \rho(s,.)^* \Omega =
\Phi^*  \Omega = \Omega_1$ for any $s \in \mathbb{T}^\tau.$ Thus $\rho_1$
preserves both $\Omega$ and $\Omega_1$.

Since $S=\{y_1=0\}$ is preserved by $\rho$, it is also preserved by $\Phi$,
which implies that it is also the singular set of $\Omega_1$, i.e.,
$\Omega_1 = \Phi^* \Omega$
is also divisible by $y_1$. Moreover, the linear part of $\Phi$ is the identity.
So we can write
$$ \Omega_1 - \Omega = y_1 \Theta, $$
where $\Theta$ is a  differential $n$-form which vanishes at $O$.

The next step is to use the equivariant Moser  path method to move $\Omega_1$
to $\Omega$ by a formal diffeomorphism without changing the torus action.
According to Lemma \ref{lem:EquivariantPath}, it is enough to show that
the equation
$$ d (Y_t \intprod (\Omega + t y_1 \Theta)) = y_1 \Theta$$
has a solution $Y_t$ (for $t \in [0,1]$).

We can write $y_1\Theta = d (y_1 \Pi)$
for some $(n-1)$-form $\Pi$ with $\Pi(O)=0$ 
and solve the following equation instead:
$$ 
Y_t \intprod (\Omega + t y_1 \Theta) = y_1 \Pi 
$$
or, equivalently,
$$ 
Y_t \intprod (dy_1 \wedge \hdots \wedge dy_n + t \Theta)) = \Pi.
$$
The equation above clearly admits a unique solution $Y_t$ because
$dy_1 \wedge \hdots \wedge dy_n$ is regular and $\Theta(O)=0$ and, moreover,
$Y_t(O)=0$ because $\Pi(O) = 0$.

Thus we have shown the existence of a simultaneous normalization of $X$ and
$\Omega$, i.e., we found a coordinate system $(y_1,\hdots,y_n)$ in which
$X^S = X^s$ (the semisimple part of $X$ coincides with the semisimple part of
its linear part) and $\Omega = y_1 dy_1 \wedge dy_2 \wedge \hdots \wedge dy_n$.
A priori, $X^s$ is not diagonal in the coordinates $(y_1,\hdots,y_n)$, so we have
a bit more work to do to diagonalize it without destroying the form
of $\Omega.$

Recall that the $(n-1)$-dimensional singular set $S = \{y_1 =0\}$ of $\Omega$ 
is invariant under the action $\rho$, so $\rho$ also preserves a line $L$ 
transversal to $S$. This line can be parametrized by $y_1$:
$$ 
L = \{(y_1, a_2 y_1, \hdots , a_n y_1)\},
$$
where $a_2, \hdots, a_n \in \mathbb{C}$ are constants.
By the linear transformation
$$(z_1, z_2, \hdots, z_n) = (y_1, y_2 - a_2y_1, \hdots,
y_n - a_ny_1)$$
we get a new coordinate system in which $\rho$ still linear,
$\Omega$ still has the same form, and both $S = \{z_1 = 0\}$ and
$L =\{z_2 = \hdots = z_n = 0\}$ are invariant with respect to $\rho$.
This means that $\rho = \rho_S \oplus \rho_L$ where $\rho_S$ is a linear
$\mathbb{T}^\tau$-action on $S$ and $\rho_L$ is a linear
$\mathbb{T}^\tau$-action on $L$. Since $\rho$ is determined by $X^S$,
we have a corresponding decomposition $X^S = X^S_S \oplus X^S_L$ for $X^S$.
Since $\Omega$ also decomposes as $\Omega = \Omega_S + \Omega_L$
with $\Omega_S = z_1 d z_1$ and $\Omega_L = dz_2 \wedge \hdots \wedge dz_n$,
the fact that $X^S$ preserves $\Omega$ implies that $X^S_S$ preserves
$\Omega_S$ and $X^S_L$  preserves $\Omega_L$. From this it is easy to see
that we can diagonalize $X^S_S$ on $S$ without destroying the form of $S$,
i.e., $\Omega_S = dx_2 \wedge \hdots \wedge dx_n$ and
$X^S_S = \sum_{i =2}^n \gamma_i x_i \dfrac{\partial}{\partial x_i}.$
Since $L$ is only 1-dimensional, $X^S_L$ is already diagonal, so we can put
$x_1 = z_1$. Then in the coordinate system $(x_1,\hdots,x_n)$ we have that
$X$ is diagonal and $\Omega$ still has the required canonical form.

This normalization is, a priori, only formal. However, if  
$\Omega$ is analytic and $X$ is analytically or Darboux integrable,
the  associated torus action $\rho$ is analytic
and all the steps above can be done analytically, so we have a local analytic
normalization. The resonance equation
\eqref{eqn:ResonanceFoldedVolume} is the same as the equation
$\mathcal{L}_{X^S} \Omega = 0$ in diagonalized normalized coordinates.
\end{proof}

\subsection{Systems preserving a singular Nambu structure of top order}
\label{subsection:LogVolume}

We start by recalling the following result, which is a simple particular case
of the linearization result of Nambu structures (see, e.g., 
\cite{DufourZung-Book}):

Consider a formal (resp. local analytic, resp. smooth) $n$-vector field
$$\Lambda = g \dfrac{\partial}{\partial x_1} \wedge \hdots \wedge
\dfrac{\partial}{\partial x_n}$$
in some coordinate system $(x_1,\hdots, x_n)$ around a point $O$
on an $n$-dimensional manifold. Such a multi-vector field of top order
is also called a \textit{\textbf{Nambu structure of top order}}.
We   say that $O$ is a \textit{\textbf{non-degenerate singular point}}
of $\Lambda$
if $g(O) = 0$ and $dg(0) \neq 0$. Clearly, this definition does not
depend on the choice of the coordinate system at $O$.

\begin{lemma}
\label{lem:LogVolume}
Let $\Lambda$ be a formal (resp. analytic, resp. smooth)  Nambu structure
of top order in dimension $n \geq 2$ with a non-degenerate singular point $O$.
 Then there exists a formal (resp. analytic, resp. smooth) coordinate system 
 $(x_1,\hdots,x_n)$ around $O$ in which we have
\begin{equation}
\label{eqn:LogVolumeCanonical}
\Lambda = x_1 \dfrac{\partial}{\partial x_1} \wedge
\dfrac{\partial}{\partial x_2} \wedge\hdots \wedge
\dfrac{\partial}{\partial x_n}.
\end{equation}
\end{lemma}

Consider now a vector field $X$ which preserves
$\Lambda = x_1 \dfrac{\partial}{\partial x_1} \wedge
\dfrac{\partial}{\partial x_2} \wedge\hdots \wedge
\dfrac{\partial}{\partial x_n}$ near $O$. First, we prove the
following simple result if $X(O) \neq 0$.

\begin{proposition}
\label{prop:LogVolumeXRegular}
If $X(O) \neq 0$ then there is a local coordinate system in which $X$ can 
be rectified together with $\Lambda$:
\begin{equation}
\label{eqn:LogVolumeXRegular}
 \Lambda = x_1 \dfrac{\partial}{\partial x_1} \wedge \hdots \wedge
\dfrac{\partial}{\partial x_n}\quad \text{and}\quad
X =\dfrac{\partial}{\partial x_n}.
\end{equation}
\end{proposition}

\begin{proof} Since $X(O) \neq 0,$ we can rectify $X$ as
$X = \dfrac{\partial}{\partial x_n}$ in a coordinate system
$(z_1,\hdots, z_{n-1}, x_n)$. The invariance of $\Lambda$ with respect to
$X$ means that it has the form $\Lambda = f(z_1,\hdots, z_{n-1})
\dfrac{\partial}{\partial z_1}\wedge \hdots \wedge
\dfrac{\partial}{\partial z_{n-1}} \wedge X = \Pi \wedge X$,
where $\Pi = f(z_1,\hdots, z_{n-1})
\dfrac{\partial}{\partial z_1}\wedge \hdots \wedge
\dfrac{\partial}{\partial z_{n-1}}$. By Lemma \ref{lem:LogVolume},
we can write $\Pi = x_1 \dfrac{\partial}{\partial x_1} \wedge \hdots \wedge
\dfrac{\partial}{\partial x_{n-1}}$ by a change of coordinates from
$(z_1,\hdots, z_{n-1})$ to $(x_1,\hdots, x_{n-1})$ which does not involve
$x_n$. Then \eqref{eqn:LogVolumeXRegular}
is satisfied in the new coordinate system $(x_1,\hdots, x_{n-1}, x)$.
 \end{proof}

Consider now the case when $X(O) = 0$.

\begin{theorem}
\label{thm:NFLogVolume}
Let $X$ be a formal vector field which vanishes at the origin
$O$ and which preserves a Nambu structure of top order $\Lambda$
with a non-degenerate singularity at $O$. Then
$X$ can be formally normalized together with $\Lambda$: there is
a formal coordinate system $(x_1,\hdots,x_n)$ in which $\Lambda$
has canonical form, $X$ is in normal form, and the semisimple part of $X$ 
is diagonal in this coordinate system
(over $\mathbb{C}$):
\begin{equation*}
\label{eqn:NFLogVolume}
\Lambda = x_1 \dfrac{\partial}{\partial x_1} \wedge \hdots \wedge
\dfrac{\partial}{\partial x_n},\quad X^S = 
\sum \gamma_i x_i \dfrac{\partial}{\partial x_i}.
\end{equation*}
The eigenvalues of $X$ satisfy the following resonance relation:
\begin{equation*}
\label{eqn:ResonanceLogVolume}
\gamma_2 + \hdots + \gamma_n = 0.
\end{equation*}
Moreover, if $X$ and $\Lambda$ are analytic and $X$ is analytically
or Darboux integrable, then there exists such a normalization
which is locally analytic.
\end{theorem}

\begin{proof}
The proof is quite similar to the case of a folded volume form.
We start with a coordinate system in which $\Lambda$ is in canonical
form \eqref{eqn:LogVolumeCanonical} but the associated torus action $\rho$
of $X$ is, a priori, nonlinear. Since $\Lambda$ is homogeneous and $\rho$
preserves $\Lambda$, the linear part $\rho_1$ of $\rho$, which is a linear
torus action, also preserves $\Lambda$. By Bochner's averaging formula,
we find a formal diffeomorphism $\Psi$ whose linear part is the identity and 
which intertwines $\rho$ with $\rho_1$. Then $\rho_1$ preserves both $\Lambda$
and $\Lambda_1 = \Psi_* \Lambda$. Similar  to the case of a folded volume 
form, $\Lambda$ and $\Lambda_1$ have the same singular set $S$, which is 
a linear subspace of codimension 1 in some coordinate system in which 
$\Lambda$ is linear and the linear part of $\Lambda_1$ coincides with 
$\Lambda$, i.e., we can write
\begin{align*}
\Lambda &= x_1 \dfrac{\partial}{\partial x_1} \wedge \hdots \wedge
\dfrac{\partial}{\partial x_n}, \\
\Lambda_1 &= x_1 \dfrac{\partial}{\partial x_1} \wedge \hdots \wedge
\dfrac{\partial}{\partial x_n} + x_1 \Pi,
\end{align*}
where $\Pi$ is an $n$-vector field which vanishes at $O$. To apply 
Lemma \ref{lem:EquivariantPath}, we must find a time-dependent vector 
field $Y_t$ such that
$$
\mathcal{L}_{Y_t}\left( x_1 \dfrac{\partial}{\partial x_1} \wedge \hdots \wedge 
\dfrac{\partial}{\partial x_n} + t x_1 \Pi\right)  = x_1 \Pi.
$$
In order to simplify the equation above, we can impose $Y_t(x_1) = 0$
(i.e., $x_1$ is a first integral for $Y_t$)  and define $\Omega_t = 
\left(\dfrac{\partial}{\partial x_1} \wedge \hdots \wedge 
\dfrac{\partial}{\partial x_n} + t \Pi\right)^{-1}$, which is a regular 
volume form since $P(O) = 0$. Put $g_t = \langle \Pi, \Omega_t \rangle$.
Then the above equation, after pairing both sides by
$- \dfrac{\Omega_t \langle \Omega_t, \cdot \rangle}{x_1}$, becomes
$$\mathcal{L}_{Y_t} \Omega_t = - g_t \Omega_t,$$
which can be easily solved. The rest of the proof is similar 
to that of Theorem \ref{thm:NFFoldedVolume}.
\end{proof}

\section{Systems with a singular symplectic structure}
\label{section:SingularSymplectic}

In this section we consider two kinds
of singular symplectic structures, namely the so called \textit{folded 
symplectic structures} (which have vanishing components) and
\textit{log-symplectic structures} (which have poles). We denote by
$ \mathbb{K} $ either $ \mathbb{R} $ or $ \mathbb{C} $.

\subsection{Folded symplectic structures}

Consider a (formal, or local analytic, or smooth)
closed 2-form $\omega$  on $(\KK^{2n},O)$, such that $\omega^n$
is non-trivial on $\KK^{2n}$ but $\omega^n (O) = 0$.
Assume that the corank of $\omega$ at $O$ is $2$, and that
$\omega^n = f dx_1 \wedge \hdots \wedge dx_{2n}$ in some local coordinate 
system, where $f(O) = 0$ but $df(O)\neq 0$,
so that $S =\{f =0\}$ is the hypersurface of singular points
of $\omega$. Assume, moreover, that the pull-back of $\omega$ to $S$ has
constant rank equal to $2(n-1)$. Such a closed 2-form $\omega$ is called
a \textit{\textbf{folded symplectic structure}} (because it can be obtained 
as the pull-back of a usual symplectic form
by a fold map \cite{Melrose-Folded1981}); it is known to admit the following
(formal, or local analytic, or smooth) normal form,
according to a classical result of Martinet \cite{Martinet-Formes1970}:
\[
 \omega=x_1dx_1\wedge dx_2+\sum_{j=1}^{n-1} dy_i\wedge dy_{j+n-1}.
\]

More generally, one may consider a \textit{\textbf{multi-folded
symplectic structure}} $\omega$, which
for some positive number $m \leq n$
admits the following canonical Darboux-like expression in some
coordinate system
$(x_1,\hdots,x_{2m},y_1,\hdots,y_{2(n-m)})$:
\begin{equation}\label{eq:mfolded-symplectic}
 \omega= \sum_{i=1}^m x_idx_i\wedge dx_{m+i} + 
 \sum_{j=1}^{n-m} dy_j\wedge dy_{j+n-m}.
\end{equation}

The multi-folded symplectic structure above is quasi-homogeneous: we have
\[
\mathcal{L}_E \omega = 6 \omega,
\]
where $E$ is the quasi-Euler vector field 
\[
E = 2\sum_{i=1}^{2m} x_i\dfrac{\partial}{\partial x_i}
+ 3\sum_{j=1}^{2(n-m)}  y_j\dfrac{\partial}{\partial y_j}.
\]
In other words, if we declare that each $x_i$ is of
quasi-homogeneous order $2$  and each $y_j$ is of quasi-homogeneous 
order $3$, then $\omega$ becomes quasi-homogeneous of order 6.

Notice that the kernel of $\omega$ at $O$ is the $2m$-dimensional vector 
space
$$
K = {\rm Span}\left\{ \dfrac{\partial}{\partial x_1},\hdots, 
\dfrac{\partial}{\partial x_{2m}}\right\} \subset \KK^{2n}.
$$
The quasi-Euler vector field $E$ is tangent to $K$ and is equal to 2 times 
the usual Euler vector field on $K$. $E$ also projects to 3 times the 
usual Euler vector field on the quotient space $\KK^{2n}/K$.

Let us assume now that there is a local (smooth, analytic, or formal) action 
$\rho$ of a compact Lie group $G$ on $(\KK^{2n},O)$  which fixes the origin $O$
and preserves $\omega$.  A priori, $\rho$ does not preserve the quasi-Euler 
vector field $E$, but it must preserve the kernel space $K$ because it 
preserves $\omega$. Denote by
$$ 
E_\rho = \int_{g \in G} \rho(g)_* E\ d\mu_G
$$
the average of $E$ with respect to $\rho$ (where $d\mu_G$ denotes the
Haar measure on $G$).

The vector field $E_\rho$ is preserved by $\rho$  and we still have 
$\mathcal{L}_{E_\rho} \omega = 6 \omega$. Moreover,
$E_\rho$ is still tangent to $K$ and has its linear part equal to 2 times 
the usual Euler vector field on $K$; the projection of the linear part 
of $E_\rho$ on the quotient space $\KK^{2n}/K$ is still equal to 3 times 
the usual Euler vector field on this quotient space. All of this implies, in 
particular, that the linear part  $E_\rho^{(1)}$ of $E_\rho$
is, a priori, non-diagonal but diagonalizable, with only two distinct 
eigenvalues 2 and 3: to diagonalize
$E_\rho^{(1)}$ we must apply a linear transformation to
$\KK^{2n}$, which is, a priori, not identity, but which is the
identity on $K$ and also projects to the identity map on
$\KK^{2n}/K$. In other words, $E_\rho^{(1)}$ is diagonal
in a coordinate system $(x_i', y_j')$ of the type
$$y_j' = y_j\ \forall j\leq 2(n-m) \quad \text{and}\quad
x_i' = x_i + \sum a_{ij} y_j\ \forall i \leq m$$
for some constants $a_{ij}$.

Notice that such a vector field is in the Poincar\'e domain and 
is non-resonant in the sense of Poincar\'e-Dulac, i.e., 
there is no resonance relation of the type 
$\gamma_i = \gamma_{j_1} + \dots + \gamma_{j_k}$ (with $k \geq 2$) among 
its eigenvalues. It follows from the classical theory of normalization of 
vector fields that $E_\rho$ is locally diagonalizable (smoothly if $E_\rho$ 
and $\rho$ are smooth, analytically if $E_\rho$ and $\rho$ are analytic, 
and formally if $E_\rho$ and $\rho$ are formal), also in an equivariant way. 
(For the smooth case, see \cite{BeKo-Equivariant2002};
the formal and analytic cases are simpler and can be done using the toric 
approach.) In other words, there is another coordinate system
$(u_1,\hdots, u_{2m},v_1,\hdots, v_{2(n-m)})$ on $(\KK^{2n},O)$, such that 
$u_i = x_i' + h.o.t.$ and $v_i = y_i' + h.o.t.$ ($h.o.t.$ means higher order 
terms, i.e., terms of degree at least 2 here), in which the action $\rho$ is 
linear, and the vector field $E_\rho$ has the following diagonal form:
\[
E_\rho = 2 \sum_{i=1}^{2m} u_i\dfrac{\partial}{\partial u_i}
+ 3 \sum_{j=1}^{2(n-m)} v_j\dfrac{\partial}{\partial v_j}.
\]
Since the linear action $\rho$ (in the coordinates
$(u_i,v_j)$) preserves $E_\rho$, it must also preserve the
eigenspaces $U$ and $V$ of $E_\rho$, where $U$ is the space where all the 
coordinates $v_i$ vanish and $V$ is the space where all the coordinates
$u_i$ vanish, and we have the splitting $\KK^{2n} = U \oplus V$.

Remember that $\mathcal{L}_{E_\rho}(\omega) = 6 \omega$. Now, $E_\rho$ is
diagonal, and the monomial 2-forms are eigenvectors of the Lie derivative 
operator $\mathcal{L}_{E_\rho}$. The only monomial 2-forms which have 
eigenvalue equal to 6 are $dv_i \wedge d v_j$ and $u_k du_i \wedge d u_j$, 
so $\omega$ must be a linear combination with constant coefficients of 
these monomial forms. (By standard division techniques, one can show easily 
that this is also true in the smooth case). It follows that $\omega$ admits 
an equivariant splitting 
$$ 
\omega = \omega_U + \omega_V,
$$
where $\omega_V$ is defined on $V$ and is constant, $\omega_U$ is defined 
on $U$ and is linear, and $\rho$ preserves both $\omega_U$ and $\omega_V$ 
(because it acts separately on $U$ and $V$ and cannot mix up these terms).

Recall that, by construction, we have
$$ 
x_i = u_i - \sum_{j=1}^{2(n-m)} a_{ij}v_j + h.o.t.
$$
and
$$ 
y_j = v_j  + h.o.t.
$$
for every $i \leq 2m$ and every $j \leq 2(n-m)$. Putting these
into the original formula \eqref{eq:mfolded-symplectic}
of $\omega$, we get 
\begin{align*}
&\omega = \\& \sum_{i=1}^m ( u_i - \sum_{j=1}^{n-m} a_{ij}v_j + h.o.t.) 
d (u_i - \sum_{j=1}^{n-m} a_{ij}v_j + h.o.t.) \wedge d(u_{i+m} - 
\sum_{j=1}^{n-m} a_{i+m,j}v_j + h.o.t.)\\
&\quad + \sum_{j=1}^{n-m} d(v_j  + h.o.t.)\wedge d (v_{j+n-m}   + h.o.t.).
\end{align*}
In the expression above, the sum of linear $U$-terms
(i.e., linear terms that do not contain
any $v_j$ or $dv_j$ in their expression) is
$\sum_{i=1}^m u_idu_i\wedge u_{m+i}$. Since
$ \omega = \omega_U + \omega_V,$ with
$\omega_U$ being a sum of linear $U$-terms, we automatically
have $\omega_U = \sum_{i=1}^m u_idu_i\wedge u_{m+i}$. By a similar argument,
 we also automatically have
$\omega_V = \sum_{j=1}^{n-m} dv_j \wedge dv_{j+n-m}.$ Thus we have proved the 
following equivariant normalization theorem for multi-folded
symplectic structures.

\begin{theorem}\label{thm:folded-symplectic}
Consider a multi-folded symplectic structure $\omega$ written as
$$ 
\omega= \sum_{i=1}^m x_idx_i\wedge dx_{m+i} + 
\sum_{j=1}^{n-m} dy_j\wedge dy_{j+n-m}
$$
($1 \leq m \leq n$) in some local (analytic, resp. formal, resp. smooth) 
coordinate system $(x_1,\hdots,x_{2m},y_1,\hdots,y_{2(n-m)})$
on $(\KK^{2n},O)$. Assume that there is a local (analytic, resp. formal, 
resp. smooth) action $\rho$ of a compact Lie group $G$ on $(\KK^{2n},O)$ which 
fixes the origin $O$ and preserves $\omega$. Then there exists a local
(analytic, resp. formal, resp. smooth) coordinate system
$(u_1,\hdots,u_{2m},v_1,\hdots,v_{2(n-m)})$
on $(\KK^{2n},O)$, in which the action $\rho$ is linear and
the multi-folded symplectic form $\omega$ is still in canonical form:
$$ 
\omega= \sum_{i=1}^m u_idu_i\wedge du_{m+i} + 
\sum_{j=1}^{n-m} dv_j\wedge dv_{j+n-m}.
$$
\end{theorem}

As an immediate consequence of the theorem above, we get the following result 
about  normalization of a vector field with an underlying (multi-)folded 
symplectic structure.

\begin{theorem}
Let $X$ be a formal vector field on $(\KK^{2n},O)$
which vanishes at $O$ and preserves a (multi-)folded
symplectic structure $\omega$ with a canonical expression
$\omega= \sum_{i=1}^m x_idx_i\wedge dx_{m+i} + 
\sum_{j=1}^{n-m} dy_j\wedge dy_{j+n-m}$  in some coordinate system.
Then the couple $(X,\omega)$ admits a formal simultaneous normalization, 
i.e., there is a formal coordinate transformation which normalizes $X$ 
while keeping the canonical expression of $\omega$. Moreover, if everything 
is locally analytic (real or complex), and $X$ is analytically   
or Darboux integrable, then this normalization can be made locally analytic.
\end{theorem}

\begin{proof}
Just apply Theorem \ref{thm:folded-symplectic} to the  associated torus 
action of $X$ at $O$.
\end{proof}

\subsection{Log-symplectic structures}

Recall that a meromorphic differential form
$\omega$ on a complex manifold $M$, with poles along a
divisor $D  \subset M$ (without  multiplicities),
is  called \textit{\textbf{logarithmic}} if both forms
$\omega$ and $d\omega$ have poles only along the divisor $D$ and
of order not greater than 1.  Such logarithmic differential forms
have been studied by many people in algebraic geometry, see, e.g.,
Deligne \cite{Deligne1970}, Saito \cite{Saito-Log1980},
Aleksandrov \cite{Aleksandrov-Log2017}. Roughly speaking, the
logarithmic condition means that if $\omega$ contains a local pole
$1/h$, where $h$ is a local holomorphic function
whose zero locus is a component of $D$ (without multiplicities),
then this pole must come together with $dh$, i.e., one can write a local
expression of $\omega$ which contains only holomorphic terms and
meromorphic terms of the type $dh/h =d(\log h)$;
otherwise $d\omega$ would contain a pole of order 2 at the zero locus of $h$.

Consider a local logarithmic differential 2-form $\omega$. Then it admits 
the following expression:
$$ 
\omega = \sum_{i,j} g_{ij} \dfrac{dh_i}{h_i} \wedge \dfrac{dh_j}{h_j}
+ \sum_k \dfrac{dh_k}{h_k}\wedge \beta_k + \gamma,
$$
where $g_{ij}, \beta_k, \gamma$ are local holomorphic functions, 1-forms, and
2-forms, respectively. The general case with non-trivial terms of the type
$g_{ij} \dfrac{dh_i}{h_i} \wedge \dfrac{dh_j}{h_j}$ is very interesting but 
also complicated, so here we will be less ambitious  and study only   
local logarithmic 2-forms without such terms. We introduce the following 
definition.

\begin{definition}
A local differential 2-form of the type
\[
\omega = \sum_{i=1}^k \dfrac{dh_i}{h_i}\wedge \beta_i + \gamma
\]
on $(\mathbb{K}^{2n},O)$ is called a
\textbf{\textit{simple log-symplectic form}}
if it satisfies the following conditions:

i) The local functions $h_1, \hdots, h_k$ vanish at the origin $O$, but
$dh_1 \wedge \hdots \wedge dh_k (O) \neq 0.$

ii) $d\omega = 0$.

iii) $h \omega^n$ is a local regular volume form, i.e.,
$h \omega^n (O) \neq 0,$ where $h = \prod_{i=1}^k h_i$ and $\omega^n$ means the
wedge product of $n$ copies of $\omega$.
\end{definition}

The definition above is valid in many different categories: real analytic,
holomorphic, formal, and smooth. Log-symplectic manifolds, especially in
the case with $k=1$ (i.e., the set of singular points is a smooth 
hypersurface), have been studied by many authors from different points 
of view, where they may be called by other names as well, such as  
$b$-symplectic or $b$-Poisson (the case with $k=1$), or $c$-symplectic,
see, e.g., \cite{GMP-bSymplectic2014,GuLi-LogSymplectic2014,
GLPR-LogSymplectic2017}.

It turns out that non-degenerate log-symplectic structures admit
\textbf{\textit{Darboux-like local normal forms}}.

\begin{theorem}\label{thm:log-symplectic1}
Let $\omega$ be a local (analytic, formal, or smooth) simple log-symplectic
structure in $(\KK^{2n},O)$ whose divisor has $k$ components.
Then there exists a local (analytic, formal or smooth) coordinate system
$(x_1,y_1,\hdots,x_k,y_k,z_1,\hdots,  z_{2(n-k)})$  in which $\omega$
has the following canonical expression:
\[
\omega = \sum_{i=1}^k \dfrac{dx_i}{x_i}\wedge d y_i + 
\sum_{j=1}^{n-k} dz_j \wedge d z_{j+n-k}.
\]
In particular, the dual Poisson structure $\Pi$ of $\omega$ is without poles
and admits the following canonical form:
\[
\Pi = \sum_{i=1}^k  {x_i} \dfrac{\partial}{\partial x_i} \wedge
\dfrac{\partial}{\partial y_i}+ \sum_{j=1}^{n-k} 
\dfrac{\partial}{\partial z_j} \wedge
\dfrac{\partial}{\partial z_{j+n-k}}.
\]
\end{theorem}

\begin{proof}
Our proof is based on the path method and consists of four small steps.

\underline{Step 1}.
We begin with a log-symplectic form $\omega$ written as
$$ \omega = \sum_{i=1}^k \dfrac{d h_i}{h_i} \wedge \beta_i + \gamma,$$
where $\beta_i$ are regular 1-forms and $\gamma$ is a regular 2-form.
The closedness condition on $\omega$ gives
$$
0 = d\omega = \sum_{i=1}^k \dfrac{d h_i}{h_i} \wedge d\beta_i + d\gamma, 
$$
which implies that, for each $i \leq k$, $dh_i \wedge d\beta_i$ is divisible 
by $h_i$. In turn, this easily implies, by Poincar\'e's Lemma, that $\beta_i$ 
can be written as $\beta_i = h\beta_i' + g dh_i + dy_i$.
Indeed, we can write $\beta_i = h_i\beta_i' + g dh_i + \xi_i$ in a coordinate 
system $(h_i, v_1,\hdots,v_{2n-1})$, in which $\xi_i$ is the part of 
$\beta_i$ that is independent of $h_i$ and does not contain $dh_i$.
The condition that $dh_i \wedge d\beta_i$ is divisible by $h_i$ means that 
$dh_i \wedge d\xi_i$ is divisible by $h_i$, which  then implies that
$d\xi_i = 0$ (otherwise $dh_i \wedge d\xi_i$ would be non-trivial and not 
divisible by $h_i$). So, by Poincar\'e's Lemma, we can write
$\xi_i = dy_i$. We can forget the term $g dh_i$ in the expression of 
$\beta_i$ because its wedge product with $\dfrac{d h_i}{h_i}$ is zero. 
Hence we can assume that $\beta_i = h_i\beta_i' + dy_i$. Notice that
$\dfrac{d h_i}{h_i} \wedge h\beta_i$ is regular and can be added to 
$\gamma$, so we have
$$ 
\omega = \sum_{i=1}^k \dfrac{d h_i}{h_i} \wedge dy_i + 
(\gamma + \sum_i dh_i \wedge \beta_i')
= \sum_{i=1}^k \dfrac{d h_i}{h_i} \wedge dy_i + \mu
$$
where $y_i$ are local functions (formal, analytic, or smooth)
and $\mu$ is a 2-form.

The value of $h\omega^n$ at $O$, where $h = \prod_{i=1}^k h_i$ and $\omega^n$ 
means the wedge product of $n$ copies of $\omega$, is equal to
$$ 
\dfrac{n!}{(n-k)!} \wedge_{i=1}^k (d h_i \wedge dy_i) \wedge \mu^{n-k} (O)
$$
if $k \leq n$, and is equal to 0 if $k > n$. This implies that $k \leq n$  
and  $ \wedge_{i=1}^k (d h_i \wedge dy_i) \wedge \mu^{n-k} (O)$ is a 
regular volume form. In particular, $(h_1,y_1,\hdots, h_k,y_k)$ can be 
completed to a regular local coordinate system $(h_1,y_1,\hdots, 
h_k,y_k, z_1, \hdots, z_{2n-2k})$.

\underline{Step 2}. Write $\mu$ as
$$ 
\mu = \sum_i dh_i \wedge d\phi_i - \sum_i dy_i \wedge d\psi_i + \mu_0 + \mu_1
$$
where $\phi_i$ and $\psi_i$ are linear functions in the coordinates
$(h_1,y_1,\hdots, h_k,y_k, z_1, \hdots, z_{2n-2k})$, $\mu_0$ is a constant 
2-form in these coordinates which contains only terms $dz_i \wedge dz_j$, 
and $\mu_1$ is a 2-form which vanishes at the origin. Then $\omega$ can be 
written as
$$
\omega = \sum_{i=1}^k \dfrac{d h_i}{h_i} \wedge dy_i + \mu
=  \sum_{i=1}^k\left( \dfrac{d h_i}{h_i} + d\psi_i\right) \wedge 
d(y_i + h_i\phi_i) +  \mu_0 + \mu_1',
$$
where $\mu_2$ is a 2-form which vanishes at $O$. In other words, we have
$$
\omega = \sum_{i=1}^k \dfrac{d h_i}{h_i} \wedge dy_i + \mu
=  \sum_{i=1}^k \dfrac{d \tilde{h}_i}{\tilde{h}_i} \wedge d\tilde{y}_i  
+ \mu_0 + \mu_1',
$$
where $\tilde{h}_i = h_i \exp(\psi_i)$ and $\tilde{y_i} = y_i + h_i\phi_i$.
Notice that $\tilde{h}_i  = h_i + h.o.t.$ (higher order terms) and
$\tilde{y}_i  = y_i + h.o.t.$, so
$(\tilde{h}_1,\tilde{y}_1,\hdots, \tilde{h}_k,\tilde{y}_k, z_1, \hdots, 
z_{2n-2k})$ is still a local coordinate system. The condition 
$h\omega^n (O) \neq 0$ 
means that $\mu_0^{n-k}(O) \neq 0$, i.e., $\mu_0^{n-k}(O)$ is a 
symplectic form of degree $2(n-k)$ in the coordinates 
$(z_1, \hdots, z_{2n-2k})$. By the Darboux theorem, by a linear
change of the coordinates $(z_1,\hdots, z_{2(n-k)})$, we may assume that
$$
\mu_0 = \sum_{j=1}^{n-k} dz_j \wedge d z_{j+n-k}.
$$

\underline{Step 3}. Put
$$ 
\omega_0 = \sum_{i=1}^k \dfrac{dh_i}{h_i}\wedge d y_i + 
\sum_{j=1}^{n-k} dz_j \wedge d z_{j+n-k}
$$
and define the following linear path of 2-forms for $t \in [0,1]$:
$$
\omega_t = t\omega_1 + (1-t)\omega_0
= \sum_{i=1}^k \dfrac{d h_{i,t}}{h_{i,t}} \wedge
d y_{i,t} + \sum_{j =1}^{n-k} dz_j \wedge d z_{j+n-k} + \mu_{1,t},
$$
where $h_{i,t} = h_i \exp (t\psi_i)$, $y_{i,t} = y_i + th_i\phi_i$, and 
$\mu_{1,t}$ is a 2-form depending polynomially on $t$ which vanishes at 
the origin.

Due to the closedness of $\omega$, we can also write
$$ 
\dfrac{d}{dt}\omega_t = \omega - \omega_0 = 
\sum_i dh_i \wedge d\phi_i - \sum_i dy_i \wedge d\psi_i +  \mu_1 = d\xi
$$
where $\xi$ is an 1-form which vanishes at the origin.

To show that $\omega$ is locally (or formally) isomorphic to the canonical 
form $\omega_0$ via a local (or formal) diffeomorphisms, we now invoke 
the path method with respect to the path $(\omega_t)$  and construct the time-1 
flow of the time-dependent vector field $X_t$ given by the (time-dependent) 
equation:
$$ 
X_t \intprod \omega_t = \xi.
$$

\underline{Step 4}. The last step is to verify that the equation above does 
admit a local (formal, analytic, or smooth) solution $X_t$ (which depends 
smoothly on $t$) such that $X_t(O) = 0.$ Indeed, the 2-form
$$
\sum_{i=1}^k \dfrac{d h_{i,t}}{h_{i,t}} \wedge d y_{i,t} 
+ dz_j \wedge d z_{j+n-k}
$$
sends the vector fields $h_{i,t}\dfrac{\partial}{\partial h_{i,t}}$,
$h_{i,t}\dfrac{\partial}{\partial y_{i,t}}$, $\dfrac{\partial}{\partial z_j}$,
$\dfrac{\partial}{\partial z_{j+n-k}}$
(written with respect to the coordinate system
$(h_{i,t}, y_{i,t}, z_j)$), via contraction, to the 1-forms (up to a sign)
$d y_{i,t}$, $d h_{i,t}$, $d z_{j+n-k}$,  $d z_{j}$. The contraction map is 
hence an isomorphism from the submodule of vector fields generated by the 
vector fields $\left\{h_{i,t}\dfrac{\partial}{\partial h_{i,t}},\,
y_{i,t}\dfrac{\partial}{\partial y_{i,t}},\,\dfrac{\partial}{\partial z_j},\,
\dfrac{\partial}{\partial z_{j+n-k}}\right\}$ to the module of 1-forms. Since
$\mu_{1,t}$ is a small perturbation term (because it vanishes at $O$), the 
contraction map by $\omega_t$ is also an isomorphism between these two modules, 
so the equation $ X_t \intprod \omega_t = \xi$ can be solved (with an unique 
solution) and, moreover, $X_t(O) = 0$ because $\xi (O) =0$.
This finishes the proof. (The $h_i$ are renamed $x_i$ at the end.)
\end{proof}

\begin{remark} The case with $k=1$ of Theorem \ref{thm:log-symplectic1} was 
obtained before by Gulliemin-Miranda-Pires \cite{GMP-bSymplectic2014}, though 
the general case (with arbitrary $k$) does not seem to have been written down 
explicitly anywhere before, as far as we know. One can also prove Theorem 
\ref{thm:log-symplectic1} differently, e.g., by the following steps: 
i) Show that the associated Poisson structure $\Pi$ of $\omega$ is smooth; 
ii)  Show that the Lie algebra which corresponds to the linear transversal part 
of $\Pi$ at the origin is isomorphic to the direct sum of $k$ copies of 
$af\!f(1)$, where $af\!f(1)$ denotes the 2-dimensional Lie algebra
of affine transformations of the line; iii) Use the result of Dufour and 
Molinier \cite{DuMo-aff1995} about the linearizability of
Poisson structures with such a linear part. Our proof of Theorem 
\ref{thm:log-symplectic1} is written in  a way that  
immediately extends to the equivariant case when there is a compact 
group action which preserves the log-symplectic form.
\end{remark}

\subsection{Equivariant normalization of log-symplectic structures}

\begin{theorem} \label{thm:log-symplectic2}
Let $\omega$ be a local (formal, analytic, or smooth) non-degenerate 
log-symplectic structure in $(\mathbb{K}^{2n},O)$ whose divisor contains 
$k$ components. Let $\rho$ be a local (formal, analytic, or smooth) action 
of a compact Lie group $G$ on $(\mathbb{K}^{2n},O)$ which fixes the origin 
$O$ and which preserves $\omega$. Then there exists a local (formal,
analytic, or smooth) coordinate system
$(x_1,y_1,\hdots,x_k,y_k,z_1,\hdots, z_{2(n-k)})$  in which the action
$\rho$ is linear and the form $\omega$
has the following Darboux-like canonical expression:
\[
\omega = \sum_{i=1}^k \dfrac{dx_i}{x_i}\wedge d y_i + 
\sum_{j=1}^{n-k} dz_j \wedge d z_{j+n-k}.
\]
\end{theorem}

\begin{proof} By Theorem \ref{thm:log-symplectic1}, we already know that, 
if we forget about the compact Lie group action, then there is a local 
coordinate system in which the dual Poisson structure $\Pi$ of $\omega$ 
has the expression
$$
\Pi = \sum_{i=1}^k  {x_i} \dfrac{\partial}{\partial x_i} \wedge
\dfrac{\partial}{\partial y_i}+ 
\sum_{j=1}^{n-k} \dfrac{\partial}{\partial w_j} \wedge
\dfrac{\partial}{\partial z_j}.
$$

Using this explicit formula, one verifies easily that $\Pi$ satisfies the 
following \textit{division property}: any local (formal, analytic, or smooth) 
vector field $X$ which is tangent to the symplectic leaves of $\Pi$ is a 
linear combination, with coefficients which are local functions, of local 
Hamiltonian vector fields of $\Pi$. A result of Miranda and Zung 
\cite{MirandaZung-Splitting2006} then says that $\Pi$ admits an 
\textit{equivariant splitting}. In fact, even without this division property, 
Poisson structures would still admit an equivariant splitting, according to 
a more recent result by Frejlich and  Mărcuţ \cite{FrMa-Normal2017}.
This means that one can split $\Pi$ equivariantly with respect to the 
$G$-action into a direct sum of two parts: the symplectic part
and the part which vanishes at zero (like in Weinstein's splitting
theorem \cite{Weinstein-Local1977}); the $G$-action is diagonal, 
acting on each factor separately. So the equivariant normalization of $\Pi$ 
amounts to an equivariant normalization of two separate parts:
the symplectic part (given by the equivariant Darboux theorem), and the part 
which vanishes at zero. Thus, this equivariant splitting result reduces the 
proof of Theorem \ref{thm:log-symplectic2} to the case when the rank of 
$\Pi$ at $O$ is zero, i.e., the case when $k=n$.

Now we can assume, by the case $k=n$ of Theorem \ref{thm:log-symplectic1},
that $\omega$ is of the form
$$
\omega = \sum_{i=1}^n \dfrac{d x_i}{x_i} \wedge dy_i
$$
in some local (or formal) coordinate system $(x_i,y_i)$
(in which the action $\rho$ of $G$ is, a priori, non-linear).
We want to linearize the action  of $G$ while keeping the above canonical 
expression of $\omega$. Notice that,  due to the homogeneity of $\omega$,
the linear part of $\rho$ (with respect to the coordinate system $(x_i,y_i)$),
denoted by $\rho_1$, is a linear $G$-action which also preserves $\omega$. 
By Bochner's local linearization theorem for compact group actions 
\cite{Bochner-Compact1945}, we know that $\rho$ is isomorphic to $\rho_1$ 
via a local (or formal) diffeomorphism
$\Phi$ whose linear part is the identity. This  means that the couple
$(\omega,\rho)$ is  isomorphic to the couple
$(\omega',\rho_1)$ via $\Phi$, where $\omega'=\Phi_*\omega$.
Due to the fact that $\rho$
must preserve the hyperplanes $S_i = \{x_i=0\}$,
we have that $\Phi_*x_i$ (obtained by Bochner's averaging formula)
is divisible by $x_i$. This implies that $\omega'$
has the form
$$
\omega' = \sum_{i=1}^n \dfrac{d x_i}{x_i} \wedge dy_i + \mu,
$$
where $\mu$ is a local holomorphic (or formal, or smooth) 2-form.
Now we can simply repeat the steps of the proof of Theorem
\ref{thm:log-symplectic1}, but in a $\rho_1$-equivariant way,
to conclude that $(\omega',\rho_1)$ is isomorphic to
$(\omega,\rho_1)$. Thus, $(\omega,\rho)$ is isomorphic to $(\omega',\rho_1)$,
which is isomorphic to $(\omega,\rho_1)$. The composed diffeomorphism which moves
$(\omega,\rho)$ to $(\omega,\rho_1)$ is our required normalization map.
\end{proof}

As an immediate consequence of Theorem \ref{thm:log-symplectic2}, we obtain
the following result.

\begin{theorem}
Let $X$ be a formal vector field on $(\KK^{2n},O)$
which vanishes at $O$ and preserves a non-degenerate log-symplectic
symplectic structure $\omega$ with a canonical expression
$\omega = \sum_{i=1}^k \dfrac{dx_i}{x_i}\wedge d y_i + 
\sum_{j=1}^{n-k} dz_j \wedge d z_{j+n-k}$ in some coordinate system.
Then the couple $(X,\omega)$ admits a
formal simultaneous normalization, i.e., there is a formal coordinate 
transformation which normalizes $X$ while keeping the canonical expression 
of $\omega$. Moreover, if everything is locally analytic (real or complex), 
and $X$ is analytically or Darboux integrable, then this 
normalization can be made locally analytic.
\end{theorem}

\begin{proof}
Just apply Theorem \ref{thm:log-symplectic2} to the  associated torus action 
of $X$ at $O$.
\end{proof}

\section{Systems with a singular contact structure}
\label{section:SingularContact}

In this section, we study generic singular contact distributions $\xi$
which are given by the kernel of singular contact forms 
$\alpha=\sum_{i=0}^{2n}f_i(x)dx_i$ on $(\KK^{2n+1}, O)$, where $O$ 
denotes the origin of $\KK^{2n+1}$.
The word ``generic" here means that $d\alpha$ is a regular
presymplectic form of rank $2n$ near $O$, i.e., $(d\alpha)^n(O)\neq 0$, 
and that, if we write $\alpha\wedge(d\alpha)^n=f dx_0\wedge 
dx_1\wedge\cdots\wedge dx_{2n}$ for some function $f$, then $f(O)=0$ 
but $df(O) \neq 0$.

We   consider the following two types  of singular contact forms 
$\alpha=\sum_{i=0}^{2n}f_i(x)dx_i$:
\begin{enumerate}
\item If $\alpha(O)\neq 0$, i.e., the kernel distribution $\xi=\ker\alpha$ 
is still regular at $O$, then $\alpha$ is called a 
\textit{\textbf{non-vanishing singular contact form}}.
\item Suppose that $\alpha(O)=0$ and, in addition, it satisfies the 
following non-degeneracy condition:
$$
F=(f_0,\ldots,f_{2n}): (\mathbb{K}^{2n+1}, O)\rightarrow 
(\mathbb{K}^{2n+1}, O) \ \text{is a local diffeomorphism}.
$$
Then $\alpha$ is called a \textit{\textbf{non-degenerate singular contact form}}.
\end{enumerate}

\begin{definition}
Let $X$ be a vector field, $\phi_X^t$ its local flow,   and $\xi$ a singular 
contact distribution defined by some singular contact form $\alpha$ on a 
$(2n+1)$-dimensional manifold $M$. We say that $X$ \textbf{\textit{preserves 
$\xi$}} if the pull back $(\phi_X^t)^*\alpha=f_t \alpha$, where $f_t$ is a 
family of smooth functions on $M$.  In particular, we say the vector field 
\textbf{\textit{preserves the singular contact form}} $\alpha$ if $f_t\equiv1$.
\end{definition}

It is easy to see that  $X$ preserves $\xi$ if and only if $\alpha$ is a 
semi-invariant of $X$, i.e., $\cL_X\alpha=
\left(\left.\dfrac{d}{dt}\right|_{t=0}f_t\right) \alpha$. The vector field 
$X$ preserves $\alpha$ if and only if $\cL_X\alpha=0$.

If a vector field $X$ preserves a singular contact form $\alpha$, then it
naturally preserves the hypersurface $S=\{x \ |\ \alpha\wedge (d\alpha)^n(x)=0\}$ 
of singular points of $\alpha$, i.e., $S$ is invariant under the flow of $X$.
In what follows, we   focus on the relatively simpler case
when $X$ preserves the singular contact form.

\subsection{Non-vanishing singular contact forms}

In this subsection, we assume that $\alpha (O) \neq0$, $(d\alpha)^n (O) \neq0$ 
and $\alpha\wedge(d\alpha)^n$ is a folded volume form. Then the $2n$-dimensional 
distribution $\xi$ is a contact distribution except on the hypersurface $S$ 
of singular points. A classical result of Martinet \cite{Martinet-Formes1970} 
states that $\alpha$ admits a Darboux normal form
\begin{equation}\label{eq:5.2}
    \alpha=\pm \theta d\theta+(x_1+1)dx_{n+1}+\sum_{k=2}^n x_kdx_{n+k},
\end{equation}
whose set of singular points is $S = \{\theta=0\}$ in these  coordinates.
In the real case, the different signs in front of the first term are important 
because they give different nonequivalent normal forms. Since  the proofs below 
are identical with both signs, we shall ignore them from now on.

\begin{lemma}\label{lemma:Contact-V.F.}
Suppose that  $X=f_0\dfrac{\partial}{\partial\theta}+
\sum_{i=1}^{2n}f_i\dfrac{\partial}{\partial x_i}$ preserves the 
singular contact form  \eqref{eq:5.2}, i.e., $\cL_X\alpha=0$. Then 
$f_0\equiv0$ and $f_1,\ldots,f_{2n}$ are independent of $\theta$.
\end{lemma}
\begin{proof}
Since $X$ preserves $\alpha$, it also preserves the  presymplectic form 
$\omega=d\alpha$, and hence it preserves the kernel 
$\KK\dfrac{\partial}{\partial\theta}$ of $\omega$. Therefore, we have
 \[
 \left[\dfrac{\partial}{\partial\theta},
 X\right]\intprod\omega=\dfrac{\partial}{\partial\theta}\intprod\cL_X\omega-
 \cL_X\left(\dfrac{\partial}{\partial\theta}\intprod\omega\right)=0,
 \]
i.e., $\left[\dfrac{\partial}{\partial\theta},X\right]$ is also in the 
kernel of $\omega$. On the other hand, we have
 \[
 \left[\dfrac{\partial}{\partial\theta},X\right]=
 \dfrac{\partial f_0}{\partial\theta}\dfrac{\partial}{\partial\theta}+
 \sum_{i=1}^n\dfrac{\partial f_i}{\partial\theta}\dfrac{\partial}{\partial x_i}.
 \]
 Therefore, $f_1,\ldots,f_{2n}$ are independent of $\theta$.

Since $0=\cL_X\alpha=X\intprod d\alpha+d(X\intprod\alpha)$, its 
$d\theta$-component  vanishes. Note that $X\intprod d\alpha$ does not 
contribute any term to the $d\theta$-component, so the $d\theta$-component 
must come from the term $d(X\intprod\alpha)$. Recall that for $i=1,\ldots,n$, 
the coefficient function in front of $dx_i$ in $\alpha$, as well as  $f_i$, 
are independent of $\theta$; therefore, this component is just 
$\dfrac{\partial(\theta f_0)}{\partial\theta}d\theta$, which implies that 
$f_0$ is identically zero.
\end{proof}

\subsubsection{The case $X(O) \neq 0$}

\begin{theorem}
Let $\alpha$ be a non-vanishing singular contact form on $(\KK^{2n+1},O)$ 
written as $\alpha= \theta d\theta+(x_1+1)dx_{n+1}+\sum_{i=2}^n x_idx_{n+i}$. 
Suppose $\alpha$ is preserved by a vector field $X$ with non-vanishing 
constant part $X^{(0)}=\sum_{i=1}^{2n}c_i\dfrac{\partial}{\partial x_i}\neq0$. 
Then $X$ can be straightened out to $X=\dfrac{\partial}{\partial x_{2n}}$, 
if $c_{n+1}=0$, or to $X=c_{n+1}\dfrac{\partial}{\partial x_{n+1}}$, if 
$c_{n+1}\neq0$, without changing the expression of $\alpha$.
\end{theorem}
\begin{proof}
We only need to consider the $1$-form  $\gamma:=\alpha-\theta d\theta$ on the 
hyperplane $\{\theta=0\}$ preserved by $X$ which can be regarded as a vector 
field on the same hyperplane by Lemma \ref{lemma:Contact-V.F.}.  First, we 
normalize the vector field $X$ by  changing only the coordinates 
$x_1,\ldots,x_{2n}$ and then  change $\gamma$ back to normal form, keeping 
the expression of $X$.

If $c_{n+1}=0$, we first assume, without loss of generality, that $c_{2n}\neq0$ 
and  write $X=\dfrac{\partial}{\partial x_{2n}}$ after some coordinate change 
by a local diffeomorphism $\Phi$. Since $\gamma$ is preserved by $X$, it is 
independent of $x_{2n}$ in the new coordinate system and hence we can  write 
it as a sum of a $1$-form $\gamma'$ on $\{x_{2n}=0\}$ and a $1$-form 
$\gamma''=g(x_1,\ldots,x_{2n-1}) dx_{2n}$ where $g$ is a function independent 
of $x_{2n}$.

We claim that $\gamma'$ on $\{x_{2n}=0\}$ has the following properties:
  \begin{itemize}
      \item $\gamma'(O) \neq 0$;
      \item $d\gamma'$ has the maximal possible rank $2n-2$ at $O$;
      \item  $\gamma'\wedge(d\gamma')^{n-1}(O) = 0.$
  \end{itemize}

In order to verify these properties, it is sufficient to calculate the constant 
part $\gamma'^{(0)}$ and the linear part $\gamma'^{(1)}$ of $\gamma'$ at $O$.

Notice that the linear part of the diffeomorphism $\Phi$ in the Straightening 
Out Theorem for vector fields is just $x_i\mapsto x_i-\dfrac{c_i}{c_{2n}}x_{2n}$ 
for $i=1,\ldots,2n-1$ and $x_{2n}\mapsto \dfrac{1}{c_{2n}}x_{2n}$. Denote 
the quadratic part of the old coordinate $x_{n+1}$ in new coordinate system 
(i.e., the quadratic part of the $(n+1)$-th component of $\Phi^{-1}$) by $Q$. 
Then the constant part $\gamma^{(0)}$ and the linear part $\gamma^{(1)}$ of 
$\gamma$ are, respectively,
\[
 \begin{aligned}
 \gamma^{(0)}&= d(x_{n+1}+c_{n+1}x_{2n}),\\
 \gamma^{(1)}&=\sum_{i=1}^{n-1}(x_i+c_ix_{2n})d(x_{n+i}+c_{n+i}x_{2n})+
 (x_n+c_nx_{2n})d(c_{2n}x_{2n})+dQ.
 \end{aligned}
\]
Write $Q$ as the sum of a function $Q'(x_1,\ldots,x_{2n-1})$ independent of 
$x_{2n}$ and a function $Q''$ divisible by $x_{2n}$. Remembering that 
$\gamma'$ is also independent of $x_{2n}$, we conclude that the constant 
part and the linear part of $\gamma'$ are, respectively,
\[
 \begin{aligned}
 \gamma'^{(0)}&= dx_{n+1},\\
 \gamma'^{(1)}&=\sum_{i=1}^{n-1}x_idx_{n+i}+dQ'.
 \end{aligned}
\]
Using these expressions, it is easy to see that the three properties stated 
above  hold.

It follows directly from \cite{Martinet-Formes1970}  that $\gamma'$ has a 
normal form $(x_1+1)d_{x_{n+1}}+\sum_{i=2}^{n-1}x_idx_{n+i}$ by changing 
only the coordinates $x_1,\ldots,x_{2n-1}$. Thus, $\gamma=(x_1+1)d_{x_{n+1}}+
\sum_{i=2}^{n-1}x_idx_{n+i}+g(x_1,\ldots,x_{2n-1})dx_{2n}$ and 
$\dfrac{\partial g}{\partial x_n}(O)\neq 0$ since $(d\gamma)^n(O) \neq 0$.
Then the transformation that maps $x_n$ to $g$ and preserves the other 
coordinates is indeed a local diffeomorphism not affecting the expression 
$X=  \dfrac{\partial}{\partial x_{2n}}$, since $g$ is independent of $x_{2n}$, 
and keeps $\gamma$ in normal form.

Note that although the position of the coordinate $x_i$ in $\gamma$ is not 
exactly the same for $i=1$, $2\leqslant i\leqslant n$, and $n+1\leqslant 
i\leqslant 2n$, the argument above works with very few changes. Indeed, if 
we assume at the beginning $c_k\neq0$ ($k\neq n+1)$ so that 
$X=\dfrac{\partial}{\partial x_k}$,  we obtain a  similar splitting 
$\gamma=\gamma'+\gamma''$ with     $\gamma'$ having the three properties
stated above. (Now the $dx_k$ component $\gamma''$ is independent of 
$x_k$ and $\gamma'^{(1)}=  \displaystyle\sum_{i=1,i\neq k}^{n}x_idx_{n+i}+dQ'$.) 
We only need to rename the coordinates after the final normalization of $\gamma$.

If $c_{n+1}\neq0$, we need to use a supplementary linear transformation before  
normalizing $\gamma'$. The first step is again the straightening out of $X$ 
to $c_{n+1}\dfrac{\partial}{\partial x_{n+1}}$ by a diffeomorphism $\Phi$ 
whose linear part is $x_i\mapsto x_i-\dfrac{c_i}{c_{n+1}}x_{n+1}$ for 
$i=1,\ldots,n,n+2,\ldots,2n$ and $x_{n+1}\mapsto x_{n+1}$.
We write again $\gamma=\gamma'+\gamma''$, where $\gamma'$ is a 1-form  on 
$\{x_{n+1}=0\}$ and $\gamma''$ has the expression $g(x_1,\ldots,x_n,x_{n+2},
\ldots,x_{2n})dx_{n+1}$ with $g(0)=1$. Since $\gamma'(O)=0$ in this case,
we use a linear transformation which maps $x_{n+1}$ to $x_{n+1}-x_{2n}$ and 
does not affect  the other $2n-1$ coordinates. Then $X$ has still the form 
$c_{n+1}\dfrac{\partial}{\partial x_{n+1}}$ after this linear transformation,
while the transformed $\gamma'$  no longer vanishes at $O$ (indeed, the 
constant part of $\gamma'$ is $dx_{2n}$). Moreover, we have 
$(d\gamma')^{n-1}(O) \neq 0$ since $(d\gamma)^n(O) \neq 0$, where
\[
 (d\gamma)^n=(d\gamma'+dg\wedge dx_{n+1})^n=
 (d\gamma')^n+n(d\gamma')^{n-1}\wedge dg\wedge dx_{n+1}=
 n(d\gamma')^{n-1}\wedge dg\wedge dx_{n+1}.
\]

We also verify that $\gamma'\wedge(d\gamma')^{n-1} (O)=0$, which implies, by 
\cite{Martinet-Formes1970}, that $\gamma'$ has the normal form 
$\gamma'=(x_2+1)dx_{n+2}+\sum_{i=3}^nx_idx_{n+i}$ by changing only the 
coordinates $x_1,\ldots,x_n,x_{n+2},\ldots,x_{2n}$. Indeed, the linear part 
of $\gamma$ is
\[
(x_1+c_1(x_{n+1}+x_{2n}))d(x_{n+1}+x_{2n})+\sum_{i=2}^n(x_i+c_i(x_{n+i}
+x_{2n}))d(x_{n+i}+c_{n+i}(x_{n+i}+x_{2n}))+dQ,
\]
where $Q$ is some quadratic function. The linear part $\gamma'^{(1)}$ of 
$\gamma'$ can be split into four parts: $\sum_{i=2}^{n-1}x_idx_{n+i}$, the 
$dx_{2n}$ component, a $1$-form divisible by $x_{2n}$, and a differential form 
$dQ'$ of some quadratic function $Q'$. Then $(d\gamma'^{(1)})^{n-1}$ is a wedge 
product of $dx_{2n}$ and some $(2n-3)$-form; therefore, its wedge product with 
the constant part $\gamma'^{(0)}=dx_{2n}$ vanishes.

Finally, we   conclude that  $\dfrac{\partial g}{\partial x_1}(O)\neq0$,
and therefore $x_1=g-g(O),x_2,\ldots,x_{2n}$ forms a coordinate system in 
which $X=c_{n+1}\dfrac{\partial}{\partial x_{n+1}}$
and $\gamma=(x_1+1)dx_{n+1}+\gamma'$.
Now use a linear transformation to arrive at the desired normalization.
\end{proof}

\subsubsection{The case $X(O)=0$}
\begin{theorem}
Let $\alpha$ be a non-vanishing singular contact form and $X$ a vector field
vanishing at $O$ which preserves $\alpha$. Then $X$ and $\alpha$ can be 
normalized simultaneously, i.e., there is a coordinate system in which 
$\alpha$ is in the normal form \eqref{eq:5.2} and $X$ commutes with its 
semisimple linear part.
\end{theorem}

\begin{proof}
We assume that $\alpha$ has been brought into  the normal form \eqref{eq:5.2}. 
Since $\cL_X\alpha=0$, regarding $X$ as a vector field on $\{\theta=0\}$ 
by Lemma \ref{lemma:Contact-V.F.}, we interpret $X$ as a Hamiltonian vector 
field with respect to the standard symplectic form $\omega_0=d\alpha$. Thus we 
have a Hamiltonian function $H$, with $H(O)=0$, i.e, $X\intprod\omega_0=-dH$.
Write 
$$
X=\sum_{i=1}^n\left(-\dfrac{\partial H}{\partial x_{n+i}}
\dfrac{\partial}{\partial x_{i}}+\dfrac{\partial H}{\partial x_{i}}
\dfrac{\partial}{\partial x_{n+i}}\right).
$$
Since $X(O)=0$, $H(O)=0$, and
\[
d(X\intprod\alpha)-dH=d(X\intprod\alpha)+X\intprod\omega_0=\cL_X\alpha=0,
\]
we get $X\intprod\alpha-H=0$, which, in the given coordinates, takes the form
\[
( x_1+1)\dfrac{\partial H}{\partial x_{1}}+
\sum_{i=2}^n x_i\dfrac{\partial H}{\partial x_{i}}-H=0.
\]
By the Taylor expansion, this equation provides two properties of $H$:
\begin{enumerate}
\item $H$ is independent of $x_1$, which implies that $x_{n+1}$ is a first 
integral of $X$;
\item $H$ is the sum of monomial terms of the form (up to constant coefficients) 
$\prod_{i=2}^{2n} x_i^{\ell_i}$ such that $\sum_{i=2}^n \ell_i=1$.
\end{enumerate}
In particular, the quadratic part $H^{(2)}$ of $H$ is the sum of $c_{ij}x_ix_j$ 
in which $c_{ij}$ is constant, $i\in\{2,\ldots,n\}$ and $j\in\{n+1,\ldots,2n\}$. 
Then the coefficient matrix $A$ of the linear part of $X$ is of the form $
\begin{pmatrix}
A_1&0\\0&A_2
\end{pmatrix}$. Moreover, we can keep
the standard form of $\omega_0$ and put $A_1$ into Jordan normal form
by two linear transformations on the complementary $n$-dimensional coordinate 
planes $\{x_1,\ldots,x_n\}$ and  $\{x_{n+1},\ldots,x_{2n}\}$.

Now we claim that the (formal) diffeomorphism $Id_\theta\oplus\Phi$ puts $X$ 
into Poincar\'e-Dulac normal form without changing $\alpha$, where 
$Id_\theta$ is the identity map for the coordinate $\theta$ and $\Phi$ is 
the classical symplectic (with respect to  $\omega_0$) diffeomorphism 
(constructed step by step below) on $\{\theta=0\}$ normalizing $H$ into the 
Poincar\'e-Birkhoff normal form. When $H$ is in the Poincar\'e-Birkhoff 
normal form, $X$ is also automatically in the Poincar\'e-Dulac normal form.

We only need to prove that $\Phi$ preserves the $1$-form 
$\beta:=\alpha-\theta d\theta$ on $\{\theta=0\}$. Let us verify 
this assertion by showing that $\Phi$ preserves the associated vector field 
$Z=(1+x_1)\dfrac{\partial}{\partial x_1}+
\sum_{i=2}^n x_i\dfrac{\partial}{\partial x_i}$ of $\beta$ defined by 
$Z\intprod\omega_0=\beta$. Indeed,  decomposing the quadratic part $H^{(2)}$
of $H$ into the sum of the semisimple part $H^s$ and the nilpotent part $H^n$,
we see that every resonant term  of $H$ lies in the subspace of the kernel 
of $X_{H^{s}}$. At step $k \geq 3$, we assume that all the non-resonant 
terms of $H$ of degree smaller than $k$ have been eliminated. In order 
to eliminate the non-resonant terms of degree $k$, one finds a homogeneous 
function $F^{(k)}$ of degree $k$ such that the image of $F^{(k)}$  under 
$X_{H^{(2)}}$ coincides with these non-resonant terms in $H^{(k)}$, i.e., 
$H^{(k)}-X_{H^{(2)}}F^{(k)}$ contains only resonant terms. Then the 
time-$1$ flow of $X_{F^{(k)}}$ is a symplectic diffeomorphism $\Phi^{(k)}$ 
which changes the $k$-th order terms of $H$ into normal form without changing 
terms of lower degrees. Constructing inductively, from $k=3$ to infinity, 
we get a formal symplectic diffeomorphism $\Phi$ (equal to the formal 
infinite composition of the above $\Phi^{(k)}$), which puts $H$ into   
Poincar\'e-Birkhoff normal form.

Thanks to the properties of $H$, and the above property of $H^{(2)}$ in 
particular, $F^{(k)}$ can be chosen such that every monomial term 
$\prod_{i=1}^{2n} x_i^{\ell_i}$ also satisfies $\ell_1=0$ and 
$\sum_{i=2}^n \ell_i=1$. Moreover, this construction guarantees that the 
higher order terms of $H$ in the new coordinate system after each 
transformation $\Phi^{(k)}$ have the same properties. Then $X_{F^{(k)}}$ 
acquires the form
\[
X_{F^{(k)}}=\sum_{i=1}^n\left(-\dfrac{\partial F^{(k)}}{\partial x_{n+i}}
(x_2,\ldots,x_{2n})\dfrac{\partial}{\partial x_{i}}+
\dfrac{\partial F^{(k)}}{\partial x_{i}}(x_{n+1},\ldots,x_{2n})
\dfrac{\partial}{\partial x_{n+i}}\right).
\]
Since $X_{F^{(k)}}$ commutes with the Euler vector field 
$E=\sum_{i=1}^n x_i\dfrac{\partial}{\partial x_i}$ and commutes with 
$\dfrac{\partial}{\partial x_1}$, we conclude that it commutes with the 
associated vector field $Z$ of $\beta$  and, furthermore, it preserves $\beta$. 
Consequently, $\Phi$ also preserves $\beta$.
\end{proof}

\subsection{Two types of non-degenerate singular contact forms}
We now consider singular contact forms which vanish at $O$, with a generic 
non-degeneracy condition.

If $O$ is a non-degenerate singularity of the singular contact form $\alpha$  
then it is isolated, by definition. Moreover, $O$ must be an equilibrium point 
of the vector field which preserves the singular contact structure.

Let us present normal forms of singular contact forms in  the transversal case 
and the generic tangential case obtained in \cite{JiangMinhZung-Contact2018}.

\begin{theorem} \label{thm:prenormalization}
Let $O$ be a non-degenerate singularity of a smooth (resp., real or complex 
analytic) singular contact 1-form $\alpha$ on a manifold
of dimension $2n+1$. Then there is a local smooth (resp., analytic)
coordinate system $(\theta, x_1, \hdots, x_{2n})$ in which
$\alpha$ has the expression
\[
\alpha = \theta d \theta + \gamma
\]
in the transversal case, or the expression
\[
\alpha = d(\theta^3 - x_1\theta) + \gamma
\]
in the tangential case with a generic tangency,
where $\gamma = \sum_{i=1}^{2n} g_i d x_i$
is a 1-form which is basic with respect to  $\dfrac{\partial}{\partial\theta}$, 
i.e., the functions $g_i$ do not depend on
$\theta$, and such that
\[
d\gamma = \sum_{i=1}^n d x_i \wedge d x_{i+n}
\]
is a symplectic form in $2n$ variables (which can be put into Darboux canonical 
form).

The primitive form $\gamma$ can be (formally, or analytically under some Diophantine conditions, or smoothly under hyperbolicity conditions) normalized 
(over $\mathbb{C}$) in the subspace $(x_1,\ldots,x_{2n})$ as follows: there 
are positive integers $n_1=1<n_2<\cdots<n_{k}<n_{k+1}=n+1$, and eigenvalues 
$\lambda_1, \hdots, \lambda_{2n}$ of $X$, such that $\lambda_i + 
\lambda_{n+i} = 1$ for all $i =1,\hdots, n$,
$\lambda_i = \lambda_j$ for any $n_s \leq i < j < n_{s+1}$ ($s=1,\hdots, k)$, 
and
\begin{equation}\label{eq:PrimitiveForm}
\gamma=\sum_{i=1}^k (\gamma_i +dQ_i)+dR;
\end{equation}
where:

i) $\gamma_i$ has the form
\[
\gamma_i=\sum_{j=n_i}^{n_{i+1}-1}\left(\lambda_{n_i} x_{j}dx_{n+j}
+(\lambda_{n_i}-1) x_{n+j}dx_{j}\right).
\]

ii) The functions $Q_i$ are as follows:
\begin{itemize}
\item{} If $\lambda_i\neq\frac{1}{2} $ then $Q_i=0$ or $Q_i=
\sum_{j=n_i}^{n_{i+1}-2}x_{j+1}x_{n+j}$;
\item{} If $\lambda_i=\frac{1}{2} $ then $Q_i$ belongs to one of the following four cases:
	\begin{itemize}
    \item{} $Q_i=0$;
	\item{} $Q_i=x_{n_i+n}^2$ with $n_{i+1}=n_i+1$ in this case;
	\item{} $Q_i=2\sum_{j=n_i}^{n_{i+1}-2}x_jx_{n+j+1}+(-1)^{n_{i+1}-n_i}
	x_{n_{i+1}-1}^2$;
	\item{} $Q_i=2\sum_{j=n_i}^{n_{i+1}-2}x_jx_{n+j+1}$ with $n_{i+1}-n_i\geq 3$ 
	and is an odd number.
	\end{itemize}
\end{itemize}

iii) $R$ is a function of $x_1,\ldots,x_{2n}$ whose infinite jet at $0$ contains 
only the monomial terms  $\prod x_i^{\alpha_i}$ satisfying the resonance relation
\begin{equation}
\label{eq:ResonantRelationPrimitiveForm}
\sum_{i=1}^{2n}\alpha_i \lambda_i = 1 \quad \text{with}\quad
\sum_{i=1}^{2n}\alpha_i\geq 3.
\end{equation}
\end{theorem}

Note that we reduce the normalization of $\gamma$ to the normalization of its 
associated vector field $Z$ defined by $Z\intprod\omega_0=\gamma$ in 
\cite{JiangMinhZung-Contact2018}. That is, supposing that $\gamma$
is already in normal form in the coordinate system $(\theta,x_1,\ldots,x_{2n})$, 
then  its associated vector field  $Z$ defined by  $Z\intprod \omega_0=\gamma$
is in Poincar\'e-Dulac normal form, where $\omega_0=d\gamma$ is
the canonical symplectic form on $\{\theta=0\}$.

We would like to mention a special example as a direct application of this 
reduction, which shows the existence of the simultaneous normalization of the 
dynamical system and the singular contact distribution (not only the form 
itself).
\begin{example}
The vector field
$Z_1=\dfrac{1}{2}\theta\dfrac{\partial}{\partial\theta}+Z$
and its semisimple part $Z_1^{s}=\dfrac{1}{2}\theta
\dfrac{\partial}{\partial\theta}+Z^{s}$
preserve the transversal singular structure.
In fact, $Z$ is invariant under its flow $\phi_Z^t(x)$, and we have 
$(\phi_Z^t)^*\omega=e^t\omega$ since $\cL_Z\omega=\omega$. Therefore, we 
have $(\phi_Z^t)^*\gamma=Z\intprod e^t\omega=e^t\gamma$.
 As the flow $\phi_{Z_1}^t(\theta,x)=(\theta e^{\frac{1}{2}t},\phi_Z^t(x))$ 
 and the transversal singular contact form $\alpha=\theta d\theta+\gamma$, 
 we have
 \[
 (\phi_{Z_1}^t)^*\alpha=(e^t\theta d\theta+e^t\gamma)=e^t\alpha.
 \]
 Similar computations are also true for the semisimple vector field $Z_1^{s}$.
\end{example}

\subsubsection{The transversal case}

Recall that a singular contact form
$\alpha$ in dimension $2n+1$ is called \textbf{\textit{transversal}} 
if the kernel of 
$d\alpha$ is transversal to the hypersurface
$S= \{x\ |\ \alpha \wedge (d\alpha)^n (x)= 0\}$ of singular points of $\alpha$.
This transversality condition is generic. In this subsubsection, we   study 
local dynamical systems which preserve this type of singular contact forms.

\begin{theorem}\label{thm:NFTransversalContact}
Suppose a transversal singular contact form $\alpha$ is preserved by a vector 
field $X$ which vanishes at $O$. Then $\alpha$ and $X$ can be normalized 
simultaneously, i.e., $\alpha$ is in the form \eqref{eq:PrimitiveForm} as 
in Theorem \ref{thm:prenormalization}, and $X$ is in Poincar\'e-Dulac normal 
form having $\theta$ as a first integral. Moreover, the quadratic function 
$Q = \sum Q_i$ and the higher order function $R$ in \eqref{eq:PrimitiveForm} 
are first integrals of the semisimple part of $X$.
\end{theorem}

\begin{proof}
Write $X=f_0\dfrac{\partial}{\partial\theta}+\sum_{i=1}^{2n}f_i
\dfrac{\partial}{\partial x_i}$. Then Lemma \ref{lemma:Contact-V.F.} also 
holds for non-degenerate singular contact forms with identical proof. 
Thus $\theta$ is a first integral of $X$,  $\cL_X\gamma=0$, and 
$\omega = d \gamma$ is a symplectic form on $\{ \theta = 0\}$ (see 
Theorem \ref{thm:prenormalization}).

Since $X$ is independent of $\theta$,
it can be projected to the hyperplane $\{\theta=0\}$. We will show that 
$X$ commutes with $Z$ defined by $Z\intprod\omega=\gamma$. Indeed, we have
\[
  [Z,X]\intprod\omega=Z\intprod\cL_{X}\omega-\cL_{X}Z\intprod\omega=
  -\cL_{X}\gamma=0,
\]
which implies that $[Z,X] = 0.$
Therefore, there exists a new coordinate system in which $Z$ and $X$ are both 
in Poincar\'e-Dulac normal forms. Then we use Moser's path method to normalize 
$\omega$ to its constant part $\omega^{(0)}$.
Moreover, we can choose a path such that this normalization preserves the 
semisimple parts of $Z$ and $X$. In other words, $Z$ and $X$ remain in   
Poincar\'e-Dulac normal form.

Indeed, take $\zeta={Z^s}\intprod(\omega-\omega^{(0)})$, where $Z^s$
is the semi-simple part of $Z$, and  $\omega_t=\omega^{(0)}+
t(\omega-\omega^{(0)})$. Then the path  of diffeomorphisms, from time 
$0$ to time $1$,  given by the flow of the time dependent vector field 
$Y_t$,  defined by $Y_t\intprod\omega_t=-\zeta$,  satisfies our 
requirements, because
$$
d Y_t\intprod\omega_t=-d\zeta=-\cL_{Z^s}(\omega-\omega^{(0)})
=-(\omega-\omega^{(0)})=-\dfrac{d}{dt}\omega_t.
$$
The third equality above is obtained in the following way. Since 
$\omega = d \gamma = d(Z\intprod \omega) = \cL_Z \omega$, it follows 
that $\cL_{Z^s} \omega = \omega$ by Theorem \ref{thm:ConservationProperty}.

It is shown in \cite[Lemma 2.2]{JiangMinhZung-Contact2018} that $Y_t$ 
commutes with $Z^s$. We now show that $Y_t$ also commutes with $X^s$.
Since $X$ preserves $\omega$, we have $\cL_{X^s}\omega=0$ by the fundamental 
conservation property (see Theorem \ref{thm:ConservationProperty}). Clearly, 
we also have $\cL_{X^s}\omega^{(0)}=0$ and $\cL_{X^s}\omega_t=0$.
Then the following two equations show that $Y_t$ also commutes with $X^s$:
\[
 \begin{aligned}
 &[Y_t,X^s]\intprod\omega_t=Y_t\intprod\cL_{X^s}\omega_t-
 \cL_{X^s}(Y_t\intprod\omega_t)=0-\cL_{X^s}(-\zeta)=\cL_{X^s}\zeta\\
 &0=[Z^s,X^s]\intprod(\omega-\omega^{(0)})=
 Z^s\intprod\cL_{X^s}(\omega-\omega^{(0)})-
 \cL_{X^s}(Z^s\intprod(\omega-\omega^{(0)}))=-\cL_{X^s}\zeta
 \end{aligned}
\]

Now we use a linear transformation to map $\omega^{(0)}$ into the canonical 
form $\omega_0$ and put, simultaneously, $Z^s$ in diagonal form, as in 
\cite{JiangMinhZung-Contact2018}. Then $\gamma$ is also in normal form. 
The linear transformation does not destroy the Poincar\'e-Dulac normal form 
of $X$. (Another way to look at it is to take the associated torus action
$\rho$ of the family of commuting vector fields $X$ and $Z$, which contains 
both the torus action associated to $X$ and the torus action associated to 
$Z$. A linearization of $\rho$ is a simultaneous normalization of $X$ and 
$Z$. Then take an equivariant normalization of $\omega$ with respect to $\rho$.)

The rest of the proof   is an application of Theorem 
\ref{thm:ConservationProperty}; here we give a direct proof.

We can now assume that everything is in normal form and, in particular,
$\omega = \omega_0$ is in canonical form.
Since $X$ commutes with the semisimple part $Z^{s}$ of $Z$, we have
\[
  \cL_{X}({Z^{s}}\intprod\gamma)=
  {Z^{s}}\intprod\cL_{X}\gamma-{[Z^{s},X]}\intprod\gamma=0,
\]
that is, ${Z^{s}}\intprod\gamma$ is a first integral of $X$. It is also a 
first integral of the semisimple linear vector field $X^{s}$ since $X$ is 
in Poincar\'e-Dulac normal form.

Recall that $Z \intprod \gamma = Z \intprod (Z \intprod \omega) = 0$, which
implies that $\cL_Z \gamma = Z \intprod d\gamma = Z \intprod \omega = \gamma$,
which, in turn, implies that $\cL_{Z^s} \gamma = \gamma$.
Recall also that $\gamma$ is the sum of its linear part $\gamma^{(1)}$ and
an exact differential form $dR$, where the  $2$-jet of the function $R$
vanishes at $O$. From $\cL_{Z^s} \gamma = \gamma$
we get $Z^{s}(R)=R$.

Since ${Z^{s}}\intprod\gamma = {Z^{s}}\intprod\gamma^{(1)} + Z^{s} \intprod R =
{Z^{s}}\intprod\gamma^{(1)}  + R$  is a first integral of
$X^{s}$, it follows that both its  quadratic part ${Z^{s}}\intprod\gamma^{(1)}$ 
and its higher order part $R$ are first integrals of $X^s$.  If we abuse the 
notation and define $\gamma^{s}:={Z^{s}}\intprod\omega_0$, then $\gamma^{(1)}
=\gamma^{s}+dQ$. Since ${Z^{s}}\intprod\gamma^{s}=
{Z^{s}}\intprod({Z^{s}}\intprod\omega_0)=0$, it follows that 
$\cL_{Z^s} \gamma^s = \gamma^s$, hence $\cL_{Z^s} dQ = dQ$.
We conclude that ${Z^{s}}\intprod dQ= \cL_{Z^s} Q = Q$ is also a first integral 
of $X^{s}$.
\end{proof}

\subsubsection{The tangential case}
We now assume that $\alpha$ is a tangential non-degenerate singular contact 
form, i.e., the kernel of $d\alpha$ is tangent to the hypersurface
$S = \{x\ |\ \alpha \wedge (d \alpha)^n(x) =0\}$ of singular points.
We recall that the normal form of $\alpha$ in the generic tangential case is 
$\alpha=d(\theta^3-\theta h(x))+\gamma$ with a non-degenerate function $h$ 
and normalized primitive $1$-form $\gamma$
(see Theorem \ref{thm:prenormalization}). We have a similar theorem as in 
the transversal case.

\begin{theorem}
\label{thm:NFTangentContact}
Let $\alpha=d(\theta^m-x_1\theta)+\gamma$ be a \textit{tangential} singular 
contact form  preserved by a vector field $X$ having $O$ as its equilibrium 
point, where $\gamma$ is a primitive $1$-form independent of $\theta$ and $m>2$ 
is an integer. Then $\alpha$ and $X$ can be normalized simultaneously, i.e., 
$\alpha=d(\theta^m-f(x)\theta)+\gamma$ with $\gamma$ in the normal form given 
in Theorem \ref{thm:prenormalization}  and $X$ in  
Poincar\'e-Dulac normal form having $\theta$ and $f(x)$ as its first integrals. 
Moreover, the quadratic function $Q = \sum Q_i$ and higher order function $R$ 
in \eqref{eq:PrimitiveForm} are first integrals of the semisimple part of $X$.

In particular, if $\alpha=d (\theta^3-x_1\theta)+\gamma$ is generically tangent, 
i.e., the kernel of $d\alpha$ has second order tangency with the  singular 
hypersurface,  then $\alpha$ and $X$ can be normalized simultaneously.
\end{theorem}

\begin{proof}
Assume that $X=f_0\dfrac{\partial}{\partial\theta}+\sum_{i=1}^{2n}
f_i\dfrac{\partial}{\partial x_i}$  preserves $\alpha=d(\theta^m-x_1\theta)
+\gamma$. Similar to Lemma \ref{lemma:Contact-V.F.} in the transversal case, 
there is an analogous  lemma (see Lemma \ref{lemma:Tangent-V.F.} below) in 
the tangent singular case, which states that $f_0$ and $f_1$ are 
identically zero. Therefore, we have $\cL_X{d(\theta^m-x_1\theta)}=
\cL_X\gamma=0$ and the functions $\theta$ and $x_1$ are first integrals of $X$.

Using this lemma, the rest of the proof is similar to the one in the 
transversal case:
\begin{itemize}
\item  $X$, viewed as a vector field  on $\{\theta=0\}$, commutes with the 
associated vector field  $Z$ defined by $Z\intprod d\gamma=\gamma$; thus   
$X$ and $Z$ can be put in   Poincar\'e-Dulac normal forms simultaneously;
\item use Moser's path method, and then a linear transformation, to normalize 
$d\gamma$ to the canonical symplectic form without destroying the normal 
forms of $X$ and $Z$.
\end{itemize}
Then the primitive $1$-form is automatically in normal form as in Theorem 
\ref{thm:prenormalization}. Notice that we do not change the variable $\theta$, 
so the old $x_1$ becomes a function $f(x)$ which is still a first integral 
of $X$ and $df(O)\neq0$.
\end{proof}

\begin{lemma}\label{lemma:Tangent-V.F.}
 Assume that $X=f_0\dfrac{\partial}{\partial\theta}+
 \sum_{i=1}^{2n}f_i\dfrac{\partial}{\partial x_i}$ preserves 
 $\alpha=d(\theta^m-x_1\theta)+\gamma$, $m>2$. Then $f_0=f_1=0$.
\end{lemma}
\begin{proof}
 Since $X$ preserves $\alpha$, it preserves the kernel 
 $\KK \dfrac{\partial}{\partial\theta}$ of $d\alpha=\omega$, which means 
 that $f_1,\ldots,f_{2n}$ are independent of $\theta$.

Consider the component $d\theta$ in $\cL_X\alpha$. Since 
$X\intprod d\alpha=\sum_{i=1}^{2n}f_i\dfrac{\partial}{\partial x_i}
\intprod\omega$ does not contain $d\theta$, the $d\theta$-component must come  
from $dX\intprod\alpha$, more precisely, from the part
\[
d\left(X\intprod d(\theta^m-x_1\theta)\right)=
d\left((m\theta^{m-1}-x_1)f_0-\theta f_1\right).
\]
Thus, we have
\begin{equation}\label{eq:5.19}
\left(\dfrac{\partial(m\theta^{m-1}-x_1)f_0}{\partial\theta}-f_1\right)d\theta=0.
\end{equation}
Now we can write
\begin{equation}
\label{eq:5.20}
(m\theta^{m-1}-x_1)f_0=\theta f_1+C(x)
\end{equation}
where $C(x)$ is a function independent of $\theta$. The coefficient of 
$d\theta$ in \eqref{eq:5.19} vanishes:
\begin{equation}\label{eq:5.21}
(m\theta^{m-1}-x_1)\dfrac{\partial f_0}{\partial\theta}+
m(m-1)\theta^{m-2}f_0-f_1=0.
\end{equation}
Let $\ell$ be any non-negative integer and assume $f_0$ is divisible by 
$x_1^\ell$. Then:
\begin{itemize}
     \item set $\theta=0$ in \eqref{eq:5.20} to conclude that $C(x)$ is 
     divisible by $x_1^{\ell+1}$;
     \item set $\theta=0$ in \eqref{eq:5.21} to conclude that $f_1(x)$ is 
     divisible by $x_1^{\ell+1}$;
     \item go back to equation \eqref{eq:5.20} to conclude that $f_0$ is also 
     divisible by $x_1^{\ell+1}$.
\end{itemize}
We proved inductively that $f_0$ is divisible by any power of $x_1$. Thus $f_0$ 
vanishes if it is a formal or analytic function. So do $f_1(x)$ and $C(x)$.
\end{proof}

\subsubsection{Singular contact forms with integrable systems}
We have seen that singular contact forms are not finitely determined  and we 
cannot find an a priori diffeomorphism that brings them in normal forms 
having finite order jets. However, the complexity of singular contact 
forms  diminishes if they are preserved by some vector fields. The more 
such vector fields, the less freedom for singular contact forms. So if a 
singular contact form is preserved by an integrable system, then the 
singular contact form itself  becomes simpler.

\begin{proposition}
Let $\gamma$ be a primitive $1$-form on $(\KK^{2n}, O)$. Suppose $\gamma$ 
is preserved by $p$ pairwise commuting vector fields $X_1,\ldots,X_p$ whose 
semisimple parts of the linear parts are independent almost everywhere. 
Then  $p\leqslant n$. Moreover, if $p=n$, then $\gamma=\gamma^{s}=
\sum_{i=1}^n(\lambda_ix_idx_{n+i}+(\lambda_i-1)x_{n+i}dx_i)$ is linear and 
the vector fields in the Poincar\'e-Dulac normal forms have diagonal linear 
parts.
\end{proposition}

\begin{proof}
From the proof of Theorem \ref{thm:NFTransversalContact}, we can assume that 
the primitive $1$-form and the vector fields are all in normal forms, i.e., 
$\gamma$ is as in Theorem \ref{thm:prenormalization} and $[X_i,X_j^s]=0$ for 
all $i,j$. By the Conservation Theorem \ref{thm:ConservationProperty}, 
$\gamma$ is also preserved by the semisimple parts $X_1^s,\ldots,X_p^s$ of 
the vector fields. After a linear transformation, we can assume 
$X_i^s=\sum_{k=1}^n\mu_{ik}x_k\dfrac{\partial}{\partial x_k}$ for 
$i=1,\ldots,p$.

Using again the equation $\cL_{X_i^s}\omega^{(0)}=\cL_{X_i^s}d\gamma=
d\cL_{X_i^s}\gamma=0$, where 
$$
\omega^{(0)}=\sum\limits_{1\leqslant j<k\leqslant 2n}c_{jk}dx_j\wedge dx_k
$$ 
is the constant part of the symplectic form $d\gamma$, we get
\begin{equation}
\label{eq:mu}
 \mu_{ij}+\mu_{ik}=0, \,\mbox{~if~} c_{jk}\neq0,\quad i=1,\ldots,p.
\end{equation}
Since $X_1^s,\ldots,X_p^s$ are independent almost everywhere, the rank of the 
$p\times(2n)$ matrix $(\mu_{ik})$ is $p$ and because $\omega^{(0)}$ is 
non-degenerate, among all the $c_{jk}$'s there are at least $n$ that do not 
vanish.  Thus \eqref{eq:mu} implies that $p\leqslant n$.

If the number of vector fields is maximal, i.e., $p=n$, then we have exactly 
$n$ among all of the  $c_{jk}$'s that do not vanish. The set of subscripts 
$j,k$ for the $n$ nonzero $c_{jk}$'s must be included in the set 
$\{1,2,\ldots,2n\}$ by non-degeneracy. Without loss of generality, we 
assume that $c_{j\,n+j}\neq0$ for all $j=1,\ldots,n$. By \eqref{eq:mu}, 
this implies $\mu_{ij}+\mu_{i\,n+j}=0$ for all $i=1,\ldots,p=n$.

The common (formal or analytic) first integrals are generated by the quadratic 
functions $x_jx_{n+j}$. Then any monomial term of $H$ has the expression 
$\prod_{j=1}^nx_j^{\ell_j}x_{n+j}^{\ell_j}$ since $H$ is automatically a 
first integral. By the resonance relation 
\eqref{eq:ResonantRelationPrimitiveForm} on the indices, we have
\[
 \sum_{j=1}^n(\lambda_j\ell_j+(1-\lambda_j)\ell_j)=\sum_{j=1}^n\ell_j=1.
\]
It thus follows that there is only one $\ell_j=1$ and the others are zero, 
which implies that $R=0$ since we request, a priori, that  the summands of 
$R$ have degree at least $3$. Hence, the primitive $1$-form now reads 
$\alpha=\gamma^{s}+dQ$. Moreover, the quadratic polynomial is also a 
common first integral of $X_i^{s}$'s. Thus $Q=\sum_{j=1}^n c_jx_jx_{n+j}$, 
where the $c_j$'s are constant coefficients. This implies $c_j=0$ since $Q$ 
does not contain any term of the form $x_jx_{n+j}$  in our normal form.
\end{proof}

We can get similar results for a transversal singular contact form $\alpha$ 
under the same assumptions.  By Theorem \ref{thm:NFTransversalContact}, 
we can assume that in a suitable coordinate system $(\theta,x_1,\ldots,x_{2n})$, 
$\alpha$ is in normal form  and the vector fields $X_1,\ldots,X_p$ are all in 
Poincar\'e-Dulac normal forms. By an argument similar to that in Lemma 
\ref{lemma:Contact-V.F.}, we can conclude that $X_1,\ldots,X_p$ are independent 
of $\theta$ and have $\theta$ as a common first integral. Therefore, 
$X_1,\ldots,X_p$ also preserve the primitive $1$-form $\gamma$ and we have 
$p\leqslant n$; if $p=n$, then $\gamma$ is linear. This proves the following 
result.
\begin{proposition}
Let $\alpha$ be a transversal singular contact form on $(\KK^{2n+1}, O)$. 
Suppose $\alpha$ is preserved by $p$ pairwise commuting vector fields 
$X_1,\ldots,X_p$ whose semisimple parts of the linear parts are independent 
almost everywhere. Then $p\leqslant n$. If $p=n$, then $\alpha=\theta d\theta
+\gamma^{s}$ is linear.
\end{proposition}
We note that the proposition is not true in the generic tangential case. The 
maximum possible number of such non-degenerate vector fields is $n-1$ since 
we can exclude the situation $p=n$. Otherwise, we use Lemma 
\ref{lemma:Tangent-V.F.} and then the vector fields preserve the 
primitive $1$-form $\gamma$, which guarantees that the common first 
integrals of the semisimple parts of the vector fields are functions of 
the $x_ix_{n+i}$'s. On the other hand, we request $f(x)$ in 
$\alpha=d(\theta^3-f\theta)+\gamma$ to be a first integral of the 
semisimple parts of the vector fields, which contradicts the 
non-degeneracy of $f(x)$, in the sense that $df(0)\neq0$.

\section*{Acknowledgements}

Part of this work was done during Nguyen Tien Zung's visits to
Shanghai Jiaotong University (SJTU) in 2018 and 2019 under
a ``High-End Foreign Experts Project" of China. He thanks the
members of the SJTU School of Mathematical Sciences, especially 
Prof. Xiang Zhang and Jie Hu,
for their warm hospitality and excellent working conditions.

\end{document}